%% file: quillen_a_paper.tex
\documentclass{amsart}

\usepackage[utf8]{inputenc}\usepackage{tgpagella}\usepackage{mathpazo}

\usepackage{sty/Donald}

\usepackage{sty/Alphabets}      
\usepackage{sty/Definitions}    
\usepackage{sty/PageSetup}      
\usepackage{sty/Environments}   

\begin{document}
\title{Quillen's Theorem A and the Whitehead theorem for bicategories}
\author{Niles Johnson}
\address{Department of Mathematics\\
	 The Ohio State University at Newark\\
	 1179 University Drive\\ 
	 Newark, OH 43055, USA}
       \email{johnson.5320@osu.edu\\
       yau.22@osu.edu}
\author{Donald Yau}


\begin{abstract}
  We prove a bicategorical analogue of Quillen's Theorem A.  As an
  application, we deduce the well-known result that a pseudofunctor is
  a biequivalence if and only if it is essentially surjective on
  objects, essentially full on 1-cells, and fully faithful on 2-cells.
\end{abstract}



\date{02 October, 2019}
\maketitle


\input{parts/intro}
\input{parts/background}

\input{parts/whitehead}



\input{parts/bibliography}

\end{document}

%% file: parts/intro.tex
\section{Introduction}\label{sec:intro} 

Quillen's Theorems A and B give conditions which imply that a functor
of categories $F\cn \C \to \D$ induces a homotopy equivalence,
respectively fibration, on geometric realizations of nerves
\cite{quillenKI}.  Bicategorical analogues of Quillen's Theorem B have
been discussed in \cite{chr} and depend on a notion of fibration
for bicategories; see, for example, \cite{bakovicFib,buckley}.  A biequivalence is
not necessarily a fibration, and therefore a separate treatment of
Theorem A is needed.

We prove a bicategorical Quillen Theorem A
(\cref{theorem:Quillen-A-bicat}) that generalizes to bicategories the
essential algebraic content of the original result.  Our proof of
Theorem A depends on a lax slice bicategory and a version of terminal
object we call \emph{inc-lax terminal} for \underline{in}itial
\underline{c}omponents (see \cref{definition:inc-lax-terminal}).

We describe the lax slice construction for a general lax functor of
bicategories in \cref{sec:lax-slice} and prove that it is a
bicategory.  Our construction is similar but not identical to notions
which have appeared in the literature, particularly \cite[Construction
  4.2.1]{buckley}.

\cref{sec:lax-terminal} is devoted to the definition and properties of
lax terminal objects.  Our notion of inc-lax terminal is stronger than
standard notions of terminal object in a bicategory, but our proof
shows that the slice bicategories over a biequivalence have inc-lax
terminal objects.

We state and prove the Bicategorical Quillen Theorem A
in \cref{sec:Quillen-A-bicat}.  We show that, given a lax functor
whose lax slices have the appropriate inc-lax terminal data, one can
construct a reverse lax functor together with strong transformations
between their composites and the respective identities.  We show,
moreover, that if the given lax functor is a pseudofunctor, then the
reverse lax functor we construct is an inverse biequivalence.

\cref{sec:Whitehead-bicat} contains our main application,
\cref{theorem:whitehead-bicat}: the local characterization of
biequivalences as those pseudofunctors which are essentially
surjective on objects, essentially full on 1-cells, and fully faithful
on 2-cells (see \cref{definition:es-ef-ff}).  This is a bicategorical
analogue of Whitehead's theorem for topological spaces, which says
that a continuous function between CW complexes is a homotopy
equivalence if and only if it is a weak homotopy equivalence.

The corresponding result for 1-categories is the well-known statement
that a functor (of 1-categories) is an equivalence if and only if it
is essentially surjective on objects and fully faithful on morphisms.
Proofs appear in many standard texts, e.g., \cite{maclane,riehl}.  Although
the result for bicategories is also well-known, the authors do not
know of a self-contained proof in the literature.  Our naming of
\cref{theorem:whitehead-bicat} as a Whitehead theorem follows the
thesis of Schommer-Pries \cite{schommer-pries}, which proves a
Whitehead theorem for symmetric monoidal bicategories.  The argument
in \cite{schommer-pries} first replaces each symmetric monoidal
bicategory with an equivalent 1-skeletal symmetric monoidal
bicategory, and then proves the Whitehead theorem in the 1-skeletal
case.

A proof of the general (non monoidal) Whitehead theorem can be
extracted from \cite{schommer-pries}, but our approach is different.
Our bicategorical Quillen Theorem A uses a contractibility of
fibers---in the form of inc-lax terminal objects---to construct an
inverse pseudofunctor directly.  Moreover, our proof of the
Bicategorical Whitehead Theorem \ref{theorem:whitehead-bicat} does not
depend on the Bicategorical Coherence Theorem
\cite{maclane-pare,street_cat-structures}, which asserts each
bicategory $\B$ is retract biequivalent to a $2$-category $\A$.

Our treatment also clarifies the role of choice in this result.  Every
instance of Whitehead's Theorem---whether topological or
algebraic---requires the axiom of choice.
Our method of proof isolates its use to
\cref{proposition:lax-slice-lax-terminal}, which chooses the inc-lax
terminal data for a pseudofunctor that is assumed to be essentially
surjective, essentially full, and fully faithful.  If one has
constructions of these data by some other means, our application of
the Bicategorical Quillen Theorem A yields a construction of an
inverse biequivalence.

This work makes extensive use of pasting diagrams in bicategories, and
was one of the motivations for the authors' concurrent work \cite{JYpasting},
which gives an elementary proof of a bicategorical pasting theorem.
The pasting diagrams we use below are pasting diagrams in the sense
defined there, and each has a unique composite by the Bicategorical
Pasting Theorem of \cite{JYpasting}.


%% file: parts/background.tex
\section{Background}\label{sec:background}

This section contains the bicategorical background needed for our
work.  We include full details of the definitions since they will be
used in subsequent sections.

\begin{definition}\label{def:bicategory}
A \emph{bicategory} is a tuple $\bigl(\B, 1, c, a, \ell, r\bigr)$ consisting of the following data.
\begin{enumerate}[label=(\roman*)]
\item $\B$ is equipped with a collection $\Ob(\B) = \B_0$, whose elements are called \emph{objects} in $\B$.  If $X\in \B_0$, we also write $X\in \B$.
\item For each pair of objects $X,Y\in\B$, $\B$ is equipped with a category $\B(X,Y)$, called a \emph{hom category}.
\begin{itemize}
\item Its objects are called \emph{$1$-cells}, and its morphisms are called \emph{$2$-cells} in $\B$.
\item Composition and identity morphisms in $\B(X,Y)$ are called \emph{vertical composition} and \emph{identity $2$-cells}, respectively.
\item For a $1$-cell $f$, its identity $2$-cell is denoted by $1_f$.
\end{itemize}
\item For each object $X\in\B$, $1_X : \boldone \to \B(X,X)$ is a functor, which we identify with the $1$-cell $1_X(*)\in\B(X,X)$, called the \emph{identity $1$-cell of $X$}.
\item For each triple of objects $X,Y,Z \in \B$, 
\[c_{XYZ} : \B(Y,Z) \times \B(X,Y) \to \B(X,Z)\]
is a functor, called the \emph{horizontal composition}.  For $1$-cells $f \in \B(X,Y)$ and $g \in \B(Y,Z)$, and $2$-cells $\alpha \in \B(X,Y)$ and $\beta \in \B(Y,Z)$, we use the notations
\[c_{XYZ}(g,f) = gf \andspace c_{XYZ}(\beta,\alpha) = \beta * \alpha.\]
\item For objects $W,X,Y,Z \in \B$, 
\[a_{WXYZ} : c_{WXZ} \bigl(c_{XYZ} \times \Id_{\B(W,X)}\bigr) \to c_{WYZ}\bigl(\Id_{\B(Y,Z)} \times c_{WXY}\bigr)\]
is a natural isomorphism, called the \emph{associator}.
\item For each pair of objects $X,Y\in \B$,
\[\begin{tikzcd}
c_{XYY} \bigl(1_Y \times \Id_{\B(X,Y)}\bigr) \rar{\ell_{XY}} & \Id_{\B(X,Y)}
& c_{XXY}\bigl(\Id_{\B(X,Y)} \times 1_X\bigr) \lar[swap]{r_{XY}}\end{tikzcd}\]
are natural isomorphisms, called the \emph{left unitor} and the \emph{right unitor}, respectively.
\end{enumerate}
The subscripts in $c$ will often be omitted.  The subscripts in $a$, $\ell$, and $r$ will often be used to denote their components.  The above data is required to satisfy the following two axioms for $1$-cells $f \in \B(V,W)$, $g \in \B(W,X)$, $h \in \B(X,Y)$, and $k \in \B(Y,Z)$.
\begin{description}
\item[Unity Axiom] The middle unity diagram
\begin{equation}\label{bicat-unity}
\begin{tikzcd}[column sep=small] (g 1_W)f \arrow{rr}{a} \arrow{rd}[swap]{r_g * 1_f} && g(1_W f) \arrow{ld}{1_g * \ell_f}\\ & gf
\end{tikzcd}
\end{equation}
in $\B(V,X)$ is commutative.
\item[Pentagon Axiom]
The diagram
\begin{equation}\label{bicat-pentagon}
\begin{tikzpicture}[commutative diagrams/every diagram]
\node (P0) at (90:2cm) {$(kh)(gf)$};
\node (P1) at (90+72:2cm) {$((kh)g)f$} ;
\node (P2) at (220:1.5cm) {\makebox[3ex][r]{$(k(hg))f$}};
\node (P3) at (-40:1.5cm) {\makebox[3ex][l]{$k((hg)f)$}};
\node (P4) at (90+4*72:2cm) {$k(h(gf))$};
\path[commutative diagrams/.cd, every arrow, every label]
(P0) edge node {$a_{k,h,gf}$} (P4)
(P1) edge node {$a_{kh,g,f}$} (P0)
(P1) edge node[swap] {$a_{k,h,g} * 1_f$} (P2)
(P2) edge node {$a_{k,hg,f}$} (P3)
(P3) edge node[swap] {$1_k * a_{h,g,f}$} (P4);
\end{tikzpicture}
\end{equation}
in $\B(V,Z)$ is commutative.
\end{description}
This finishes the definition of a bicategory.
\end{definition}

We make use of two additional compatibilities between unitors.  The
statements and their proofs are bicategorical analogues of
corresponding results for monoidal 1-categories; cf., \cite[Theorems 6 and 7]{kelly2}.
\begin{proposition}\label{bicat-left-right-unity}
Suppose $f\in\B(X,Y)$ and $g\in\B(Y,Z)$ are $1$-cells.  Then the diagrams
\[\begin{tikzcd}[column sep=small]
(1_Zg)f \arrow{rr}{a} \arrow{rd}[swap]{\ell_g * 1_f} && 1_Z(gf) \arrow{ld}{\ell_{gf}}\\ 
& gf & \end{tikzcd}\qquad
\begin{tikzcd}[column sep=small]
(gf)1_X \arrow{rr}{a} \arrow{rd}[swap]{r_{gf}} && g(f1_X) \arrow{ld}{1_g*r_f}\\ 
& gf & \end{tikzcd}\]
in $\B(X,Z)$ are commutative. 
\end{proposition} 
\begin{proposition}\label{bicat-l-equals-r}
For each object $X$ in $\B$, the equality
\[\begin{tikzcd}\ell_{1_X} = r_{1_X} : 1_X 1_X \rar{\cong} & 1_X\end{tikzcd}\]
holds in $\B(X,X)$.
\end{proposition}

\subsection{Lax functors, transformations, and modifications}

\begin{definition}\label{def:lax-functors}
Suppose $(\B,1,c,a,\ell,r)$ and $(\B',1',c',a',\ell',r')$ are bicategories.  A \emph{lax functor} \[(F,F^2,F^0) : \B\to\B'\] from $\B$ to $\B'$ is a triple consisting of the following data.
\begin{itemize}
\item $F : \B_0 \to \B'_0$ is a function on objects.
\item For each pair of objects $X,Y$ in $\B$, it is equipped with a functor \[F : \B(X,Y) \to \B'(FX,FY).\]
\item For all objects $X,Y,Z$ in $\B$, it is equipped with natural transformations
\[\begin{tikzpicture}[commutative diagrams/every diagram, xscale=3.7, yscale=1.5]
\tikzset{arrow/.style={commutative diagrams/.cd,every arrow}}
\node (A) at (0,1) {$\B(Y,Z)\times\B(X,Y)$}; 
\node (B) at (1,1) {$\B(X,Z)$}; 
\node (C) at (0,0) {$\B'(FY,FZ)\times\B'(FX,FY)$}; 
\node (D) at (1,0) {$\B'(FX,FZ)$}; 
\node[font=\Large] at (.6,.5) {\rotatebox{45}{$\Rightarrow$}}; 
\node at (.5,.6) {$F^2$};
\draw [arrow] (A) to node{$c$} (B); 
\draw [arrow] (B) to node{$F$} (D);
\draw [arrow] (A) to node[swap]{$F\times F$} (C); 
\draw [arrow] (C) to node[swap]{$c'$} (D);
\end{tikzpicture}\qquad
\begin{tikzpicture}[commutative diagrams/every diagram, xscale=2, yscale=1.5]
\tikzset{arrow/.style={commutative diagrams/.cd,every arrow}}
\node (A) at (0,1) {$\boldone$}; 
\node (B) at (1,1) {$\B(X,X)$}; \node (C) at (0,0) {}; 
\node (D) at (1,0) {$\B'(FX,FX)$}; 
\node[font=\Large] at (.6,.5) {\rotatebox{45}{$\Rightarrow$}}; 
\node at (.4,.6) {$F^0$};
\draw [arrow] (A) to node{$1_X$} (B);
\draw [arrow] (B) to node{$F$} (D);
\draw [arrow, out=-90, in=170] (A) to node[near start, swap]{$1'_{FX}$} (D);
\end{tikzpicture}\]
with component $2$-cells 
\[\begin{tikzcd}Fg \circ Ff \ar{r}{F^2_{g,f}} & F(gf)\end{tikzcd}\andspace
\begin{tikzcd} 1'_{FX} \ar{r}{F^0_X} & F1_X.\end{tikzcd}\]
\end{itemize}
The above data is required to make the following three diagrams commutative for all $1$-cells $f \in \B(W,X)$, $g \in \B(X,Y)$, and $h \in \B(Y,Z)$.
\begin{description}
\item[Lax associativity]
\begin{equation}\label{f2-bicat}
\begin{tikzpicture}[x=19mm,y=19mm,baseline={(a.base)}]
  \draw[0cell] 
  (0,0) node (a) {(Fh \circ Fg) \circ Ff}
  (60:1) node (b) {Fh \circ (Fg \circ Ff)}
  (-60:1) node (c) {F(hg) \circ Ff}
  (3.25,0) node (d) {F(h(gf))}
  (d) ++(120:1) node (e) {Fh \circ F(gf)}
  (d) ++(240:1) node (f) {F((hg)f)}
  ;
  \path[1cell] 
  (a) edge node {a'} (b)
  (b) edge node {1_{Fh}*F^2_{g,f}} (e)
  (e) edge node[pos=.7] {F^2_{hg,f}} (d)
  (a) edge[swap] node[pos=.2] {F^2_{h,g} *1_{Ff}} (c)
  (c) edge[swap] node {F^2_{h,gf}} (f)
  (f) edge[swap] node {Fa} (d)
  ;
\end{tikzpicture}
\end{equation}
in $B'(FW,FZ)$.
\item[Lax left and right unity]
\begin{equation}\label{f0-bicat}
\begin{tikzpicture}[x=17mm,y=16mm,baseline=(b.base)]
  \draw[0cell] 
  (0,0) node (a) {1'_{FX} \circ Ff}
  (.15,1) node (b) {F1_X \circ Ff}
  (2,0) node (d) {Ff}
  (d) ++(-.15,1) node (c) {F(1_X\circ f)}
  ;
  \path[1cell] 
  (a) edge node[pos=.3] (X) {F^0_X*1_{Ff}} (b)
  (b) edge node {F^2_{1_X,f}} (c)
  (c) edge node {F\ell} (d)
  (a) edge node {\ell'} (d)
  ;
\end{tikzpicture}
\qquad
\begin{tikzpicture}[x=18mm,y=16mm,baseline=(b.base)]
  \draw[0cell] 
  (0,0) node (a) {Ff \circ 1'_{FW}}
  (.15,1) node (b) {Ff \circ F1_W }
  (2,0) node (d) {Ff}
  (d) ++(-.15,1) node (c) {F(f\circ 1_W)}
  ;
  \path[1cell] 
  (a) edge node[pos=.3] (X) {1_{Ff}*F^0_W} (b)
  (b) edge node {F^2_{f,1_W}} (c)
  (c) edge node {Fr} (d)
  (a) edge node {r'} (d)
  ;
\end{tikzpicture}
\end{equation}
in $\B'(FW,FX)$.
\end{description}
This finishes the definition of a lax functor.  Moreover:
\begin{itemize}
\item A lax functor is \emph{unitary} (resp., \emph{strictly unitary})
  if each 2-cell $F^0_X$ is an isomorphism (resp., identity).
\item A \emph{pseudofunctor} is a lax functor in which $F^2$ and $F^0$ are natural isomorphisms.
\item A \emph{strict functor} is a lax functor in which $F^2$ and $F^0$ are identity natural transformations.  
\end{itemize}
We let $\Id_\B$ denote the identity strict functor for a bicategory $\B$.
\end{definition}

The next result defines the \emph{constant pseudofunctor} at an object.
\begin{proposition}\label{constant-pseudofunctor}
Suppose $X$ is an object in a bicategory $\B$, and $\A$ is another bicategory.  Then there is a strictly unitary pseudofunctor \[\conof{X} : \A\to\B\] defined as follows.
\begin{itemize}
\item $\conof{X}$ sends each object of $\A$ to $X$.
\item For each pair of objects $Y,Z$ in $\A$, the functor 
\[\conof{X} : \A(Y,Z) \to \B(X,X)\] sends 
\begin{itemize}
\item every $1$-cell in $\A(Y,Z)$ to the identity $1$-cell $1_X$ of $X$;
\item every $2$-cell in $\A(Y,Z)$ to the identity $2$-cell $1_{1_X}$ of the identity $1$-cell.
\end{itemize}
\item For each object $Y$ of $\A$, the lax unity constraint is 
\[(\conof{X})^0_Y = 1_{1_X} : 1_X \to 1_X.\]
\item For each pair of composable $1$-cells $(g,f)$ in $\A$, the lax functoriality constraint is 
\[(\conof{X})^2_{g,f} = \ell_{1_X} : 1_X 1_X \to 1_X.\]
\end{itemize}
\end{proposition}
The proof that these data define a strictly unitary pseudofunctor is
an exercise in the unity properties, notably
\cref{bicat-left-right-unity,bicat-l-equals-r}, and we leave it to the reader.


\begin{definition}\label{definition:lax-transformation}
Let $(F,F^2,F^0)$ and $(G,G^2,G^0)$ be lax functors $\B \to \B'$. A \emph{lax transformation} $\alpha\cn F \to G$ consists of the following data.
\begin{description}
\item[Components] It is equipped with a component $1$-cell $\alpha_X\in \B'(FX,GX)$ for each object $X$ in $\B$.  
\item[Lax naturality constraints] For each pair of objects $X,Y$ in $\B$, it is equipped with a  natural transformation
\[\alpha : \al_X^* G \to (\al_Y)_* F : \B(X,Y) \to \B'(FX,GY),\] 
with a component 2-cell $\alpha_f \cn (Gf) \al_X \to \al_Y (Ff)$, as in the following diagram,  for each 1-cell $f\in\B(X,Y)$.
\[\begin{tikzpicture}[xscale=2.5, yscale=-1.8]
\tikzset{arrow/.style={commutative diagrams/.cd,every arrow},swap}
\node (F1) at (0,0) {$FX$}; \node (F2) at (1,0) {$FY$}; 
\node (G1) at (0,1) {$GX$}; \node (G2) at (1,1) {$GY$};
\draw[arrow,swap] (F1) to node{$Ff$} (F2);
\draw[arrow] (F1) to node{$\alpha_X$} (G1);
\draw[arrow] (F2) to node[swap]{$\alpha_Y$} (G2);
\draw[arrow] (G1) to node{$Gf$} (G2);
\node[font=\Large] at (.4,.6) {\rotatebox{45}{$\Rightarrow$}}; 
\node at (.6,.6) {$\alpha_{f}$};
\end{tikzpicture}\]
\end{description}
The above data is required to satisfy the following two pasting diagram equalities for all objects $X,Y,Z$ and 1-cells $f \in \B(X,Y)$ and $g \in \B(Y,Z)$.
\begin{description}
\item[Lax unity]
\begin{equation}\label{unity-transformation-pasting}
\begin{tikzpicture}[xscale=2.5, yscale=-2, baseline=(al.base)]
\tikzset{arrow/.style={commutative diagrams/.cd,every arrow,swap}}
\node (F1) at (0,0) {$FX$}; \node (F2) at (1,0) {$FX$}; 
\node (G1) at (0,1) {$GX$}; \node (G2) at (1,1) {$GX$};
\draw[arrow, bend right] (F1) to node[swap]{$F1_X$} (F2);
\draw[arrow] (F1) to node (al) {$\alpha_X$} (G1);
\draw[arrow] (F2) to node[swap]{$\alpha_X$} (G2);
\draw[arrow,bend left] (G1) to node{$1_{GX}$} (G2);
\draw[arrow,bend right] (G1) to node[swap]{$G1_{X}$} (G2);
\node[font=\Large] at (.4,.2) {\rotatebox{45}{$\Rightarrow$}}; 
\node at (.6,.2) {$\alpha_{1_X}$};
\node[font=\Large] at (.45,1) {\rotatebox{90}{$\Rightarrow$}}; 
\node at (.6,1) {$G^0$};
\node at (1.7,.5) {\LARGE{$=$}};
\node (F3) at (2.4,0) {$FX$}; \node (F4) at (3.4,0) {$FX$}; 
\node (G3) at (2.4,1) {$GX$}; \node (G4) at (3.4,1) {$GX$};
\draw[arrow, bend right] (F3) to node[swap]{$F1_X$} (F4);
\draw[arrow] (F3) to node{$\alpha_X$} (G3);
\draw[arrow] (F4) to node[swap]{$\alpha_X$} (G4);
\draw[arrow,bend left] (G3) to node{$1_{GX}$} (G4);
\draw[arrow, bend left] (F3) to node[inner sep=0pt]{$\alpha_X$} (G4);
\draw[arrow, bend left] (F3) to node{$1_{FX}$} (F4);
\node[font=\Large] at (2.7,1) {\rotatebox{45}{$\Rightarrow$}}; 
\node at (2.8,1) {$\ell$};
\node[font=\Large] at (2.9,.6) {\rotatebox{45}{$\Rightarrow$}}; 
\node at (3.1,.6) {$r^{-1}$};
\node[font=\Large] at (2.85,0) {\rotatebox{90}{$\Rightarrow$}}; 
\node at (3,0) {$F^0$};
\end{tikzpicture}
\end{equation}
\item[Lax naturality]
\begin{equation}\label{2-cell-transformation-pasting}
\begin{tikzpicture}[xscale=1.7, yscale=-2, baseline=(al.base)]
\tikzset{arrow/.style={commutative diagrams/.cd,every arrow},swap}
\node (F1) at (0,0) {$FX$}; \node (F2) at ($(F1) + (2,0)$) {$FZ$}; 
\node (G1) at ($(F1)+(0,1)$) {$GX$}; \node (G2) at ($(F1)+(1,1.3)$) {$GY$};
\node (G3) at ($(F1)+(2,1)$) {$GZ$};
\node[font=\Large] at ($(F1)+(.8,1)$) {\rotatebox{90}{$\Rightarrow$}}; 
\node at ($(F1)+(1.1,1)$) {$G^2$};
\node[font=\Large] at ($(F1)+(.8,.1)$) {\rotatebox{45}{$\Rightarrow$}}; 
\node at ($(F1)+(1.1,.1)$) {$\alpha_{gf}$};
\draw[arrow, bend right=15] (F1) to node[swap]{$F(gf)$} (F2);
\draw[arrow] (F2) to node[swap]{$\alpha_Z$} (G3);
\draw[arrow] (F1) to node (al) {$\alpha_X$} (G1);
\draw[arrow, bend left=10] (G1) to node{$Gf$} (G2);
\draw[arrow, bend left=10] (G2) to node{$Gg$} (G3);
\draw[arrow, bend right=15] (G1) to node[swap]{$G(gf)$} (G3);
\node at ($(F1)+(2.75,.5)$) {\LARGE{$=$}};
\node (F3) at ($(F1)+(3.5,0)$) {$FX$}; 
\node (F4) at ($(F3) + (1,.3)$) {$FY$}; \node (F5) at ($(F3) + (2,0)$) {$FZ$}; 
\node (G4) at ($(F3)+(0,1)$) {$GX$}; \node (G5) at ($(F3)+(1,1.3)$) {$GY$};
\node (G6) at ($(F3)+(2,1)$) {$GZ$};
\node[font=\Large] at ($(F3)+(.3,.8)$) {\rotatebox{45}{$\Rightarrow$}}; 
\node at ($(F3)+(.5,.9)$) {$\alpha_f$};
\node[font=\Large] at ($(F3)+(1.3,.8)$) {\rotatebox{45}{$\Rightarrow$}}; 
\node at ($(F3)+(1.5,.9)$) {$\alpha_g$};
\node[font=\Large] at ($(F3)+(.9,.05)$) {\rotatebox{90}{$\Rightarrow$}}; 
\node at ($(F3)+(1.1,.05)$) {$F^2$};
\draw[arrow, bend right=15] (F3) to node[swap]{$F(gf)$} (F5);
\draw[arrow] (F5) to node[swap]{$\alpha_Z$} (G6);
\draw[arrow] (F3) to node{$\alpha_X$} (G4);
\draw[arrow, bend left=10] (G4) to node{$Gf$} (G5);
\draw[arrow, bend left=10] (G5) to node{$Gg$} (G6);
\draw[arrow, bend left=10] (F3) to node[inner sep=1]{$Ff$} (F4);
\draw[arrow, bend left=10] (F4) to node[inner sep=1]{$Fg$} (F5);
\draw[arrow] (F4) to node[near start, inner sep=1]{$\alpha_Y$} (G5);
\end{tikzpicture}
\end{equation}
\end{description}
This finishes the definition of a lax transformation.  Moreover, a
\emph{strong transformation} is a lax transformation in which every
component 2-cell $\alpha_f$ is an isomorphism.  We let $1_F$ denote
the identity strong transformation on a lax functor $F$.
\end{definition}


\begin{definition}\label{def:modification}
Suppose $\alpha,\beta : F \to G$ are lax transformations for lax functors $F,G : \B\to\B'$.  A \emph{modification} $\Gamma : \alpha \to \beta$ consists of a component $2$-cell \[\Gamma_X : \alpha_X \to \beta_X\] in $\B'(FX,GX)$ for each object $X$ in $\B$, that satisfies the following \emph{modification axiom}
\begin{equation}\label{modification-axiom}
\begin{tikzpicture}[xscale=2, yscale=-2, baseline=(eq.base)]
\tikzset{arrow/.style={commutative diagrams/.cd,every arrow},swap}
\node (F1) at (0,0) {$FX$}; \node (F2) at ($(F1) + (1,0)$) {$FY$}; 
\node (G1) at ($(F1)+(0,1)$) {$GX$}; \node (G2) at ($(F1)+(1,1)$) {$GY$};
\node[font=\Large] at ($(F1)+(.15,.5)$) {\rotatebox{45}{$\Rightarrow$}}; 
\node[font=\small] at ($(F1)+(.25,.65)$) {$\alpha_{f}$};
\node[font=\Large] at ($(F1)+(1,.4)$) {$\Rightarrow$}; 
\node[font=\small] at ($(F1)+(1,.55)$) {$\Gamma_Y$};
\draw[arrow] (F1) to node[swap]{\small{$Ff$}} (F2);
\draw[arrow, bend right=40] (F2) to node[swap]{\small{$\beta_Y$}} (G2);
\draw[arrow, bend left=40] (F1) to node{\small{$\alpha_X$}} (G1);
\draw[arrow] (G1) to node{\small{$Gf$}} (G2);
\draw[arrow, bend left=40] (F2) to node{\small{$\alpha_Y$}} (G2);
\node (eq) at ($(F1)+(2,.5)$) {\LARGE{$=$}};
\node (F3) at ($(F1)+(3,0)$) {$FX$}; \node (F4) at ($(F3) + (1,0)$) {$FY$}; 
\node (G3) at ($(F3)+(0,1)$) {$GX$}; \node (G4) at ($(F3)+(1,1)$) {$GY$};
\node[font=\Large] at ($(F3)+(.75,.55)$) {\rotatebox{45}{$\Rightarrow$}}; 
\node[font=\small] at ($(F3)+(.9,.65)$) {$\beta_{f}$};
\node[font=\Large] at ($(F3)+(0,.4)$) {$\Rightarrow$}; 
\node[font=\small] at ($(F3)+(0,.55)$) {$\Gamma_X$};
\draw[arrow] (F3) to node[swap]{\small{$Ff$}} (F4);
\draw[arrow, bend right=40] (F4) to node[swap]{\small{$\beta_Y$}} (G4);
\draw[arrow, bend left=40] (F3) to node{\small{$\alpha_X$}} (G3);
\draw[arrow] (G3) to node{\small{$Gf$}} (G4);
\draw[arrow, bend right=40] (F3) to node[swap]{\small{$\beta_X$}} (G3);
\end{tikzpicture}
\end{equation}
for each $1$-cell $f \in \B(X,Y)$.  A modification is
\emph{invertible} if each component 2-cell $\Gamma_X$ is an isomorphism.
\end{definition}

\subsection{Adjoint and invertible 1-cells}

In this section we recall basic notions of internal adjunction,
invertibility, and mates.  We will need these for the constructions in
following sections.

\begin{definition}\label{definition:internal-adjunction}
  An \demph{internal adjunction} in a bicategory
  $\sB$ is a quadruple $(f,g,\eta,\epz)$ consisting of
  \begin{itemize}
  \item 1-cells $f\cn X \to Y$ and $g\cn Y \to X$;
  \item 2-cells $\eta\cn 1_X \to gf$ and $\epz\cn fg \to 1_Y$.
  \end{itemize}
  These data are subject to the following two axioms, in the form of
  commutative triangles.
  \begin{diagram}\label{diagram:triangles}
    \begin{tikzpicture}[x=20mm,y=20mm,rotate=-45,baseline={(fgf2.base)}]
      \draw[0cell] 
      (0,0) node (f1) {f\, 1_X}
      (1,0) node (fgf1) {f(gf)}
      (2,0) node (fgf2) {(fg)f}
      (2,-1) node (1f) {1_Y\, f}
      (2,-2) node (f) {f}
      ;
      \draw[1cell] 
      (f1) edge[swap] node {r_f} (f)
      (f1) edge node {1_f * \eta} (fgf1)
      (fgf1) edge node {a_{f,g,f}^\inv} (fgf2)
      (fgf2) edge node {\epz * 1_f} (1f)
      (1f) edge node {\ell_f} (f)
      ;
      \draw[2cell] 
      
      ;
    \end{tikzpicture}
    \qquad \qquad
    \begin{tikzpicture}[x=20mm,y=20mm,rotate=-45,baseline={(gfg2.base)}]
      \draw[0cell] 
      (0,0) node (1g) {1_X\, g}
      (1,0) node (gfg1) {(gf)g}
      (2,0) node (gfg2) {g(fg)}
      (2,-1) node (g1) {g\, 1_Y}
      (2,-2) node (g) {g}
      ;
      \draw[1cell] 
      (1g) edge[swap] node {\ell_g} (g)
      (1g) edge node {\eta * 1_g} (gfg1)
      (gfg1) edge node {a_{g,f,g}} (gfg2)
      (gfg2) edge node {1_g * \epz} (g1)
      (g1) edge node {r_g} (g)
      ;
      \draw[2cell] 
      
      ;
    \end{tikzpicture}
  \end{diagram}
\end{definition}

If $(f,g,\eta,\epz)$ is an adjunction with $f\cn X \to Y$, then the
represented adjunctions given by pre- and post-composition induce
isomorphisms of 2-cells; corresponding 2-cells under these
isomorphisms are known as mates, and defined as follows.
\begin{definition}\label{definition:mates}
  Suppose $(f_0, g_0, \eta_0, \epz_0)$ and $(f_1, g_1, \eta_1,
  \epz_1)$ is a pair of adjunctions in $\B$, with $f_0\cn X_0 \to Y_0$
  and $f_1\cn X_1 \to Y_1$.  Suppose moreover that $a\cn X_0 \to X_1$
  and $b\cn Y_0 \to Y_1$ are 1-cells in $\B$.
  The \emph{mate} of a 2-cell $\omega\cn
  f_1a \to bf_0$ is given by the pasting diagram at left below.
  Likewise, the \emph{mate} of a 2-cell $\nu\cn ag_0 \to g_1b$ is given by
  the pasting diagram at right below.
  \[
  \begin{tikzpicture}[x=20mm,y=20mm,scale=.85]
    \draw[0cell] 
    (0,0) node (x0) {X_0}
    (1,-.5) node (x1) {X_1}
    (0,-1) node (y0) {Y_0}
    (1,-1.5) node (y1) {Y_1}
    (-1,0.5) node (y0') {Y_0}
    (2,-2) node (x1') {X_1}
    (y0') ++(2.6,-.2) node (T) {X_0}
    (x1') ++(-2.6,.2) node (B) {Y_0}
    ;
    \draw[1cell] 
    (x0) edge node (a) {a} (x1)
    (y0) edge[swap] node (b) {b} (y1)
    (x0) edge[swap] node {f_0} (y0)
    (x1) edge node {f_1} (y1)
    (y0') edge node {g_0} (x0)
    (y0') edge[swap,bend right=20] node[pos=.6] (1y) {1} (y0)
    (y1) edge[swap] node {g_1} (x1')
    (x1) edge[bend left=20] node[pos=.4] (1x) {1} (x1')
    (y0') edge[bend left=20] node {g_0} (T)
    (T) edge[bend left=15] node {a} (x1')
    (y0') edge[bend right=15, swap] node {b} (B)
    (B) edge[bend right=20, swap] node {g_1} (x1')
    ;
    \draw[2cell] 
    node[between=y0 and x1 at .5, rotate=225, font=\Large] (A) {\Rightarrow}
    (A) node[above left] {\omega}
    node[between=1y and x0 at .4, rotate=225, font=\Large] (E) {\Rightarrow}
    (E) node[above left] {\epz_0}
    node[between=y1 and 1x at .6, rotate=225, font=\Large] (H) {\Rightarrow}
    (H) node[below right] {\eta_1}
    node[between=T and a at .5, rotate=225, font=\large] (U) {\Rightarrow}
    (U) node[above left] {\ell^\inv}
    node[between=B and b at .5, rotate=225, font=\large] (V) {\Rightarrow}
    (V) node[below right] {r}
    ;
  \end{tikzpicture}
  \quad\qquad
  \begin{tikzpicture}[x=20mm,y=20mm, scale=.85]
    \draw[0cell] 
    (0,0) node (x0) {X_0}
    (1,0.5) node (x1) {X_1}
    (0,-1) node (y0) {Y_0}
    (1,-0.5) node (y1) {Y_1}
    (2,1) node (y1') {Y_1}
    (-1,-1.5) node (x0') {X_0}
    (x0') ++(2.6,.2) node (B) {Y_1}
    (y1') ++(-2.6,-.2) node (T) {X_1}
    ;
    \draw[1cell] 
    (x0) edge node (a) {a} (x1)
    (y0) edge[swap] node (b) {b} (y1)
    (y0) edge node {g_0} (x0)
    (y1) edge[swap] node {g_1} (x1)
    (x1) edge node {f_1} (y1')
    (y1) edge[swap, bend right=20] node[pos=.4] (1y) {1} (y1')
    (x0') edge[swap] node {f_0} (y0)
    (x0') edge[bend left=20] node[pos=.6] (1x) {1} (x0)
    (x0') edge[bend left=15] node {a} (T)
    (T) edge[bend left=20] node {f_1} (y1')
    (x0') edge[bend right=20, swap] node {f_0} (B)
    (B) edge[bend right=15, swap] node {b} (y1')
    ;
    \draw[2cell] 
    node[between=y0 and x1 at .5, rotate=-45, font=\Large] (A) {\Rightarrow}
    (A) node[above right] {\nu}
    node[between=x1 and 1y at .6, rotate=-45, font=\Large] (E) {\Rightarrow}
    (E) node[above right] {\epz_1}
    node[between=y0 and 1x at .6, rotate=-45, font=\Large] (H) {\Rightarrow}
    (H) node[below left] {\eta_0}
    node[between=T and a at .5, rotate=-45, font=\large] (U) {\Rightarrow}
    (U) node[above right] {r^\inv}
    node[between=B and b at .5, rotate=-45, font=\large] (V) {\Rightarrow}
    (V) node[below left] {\ell}
    ;
  \end{tikzpicture}
  \]
\end{definition}
The triangle identities imply that these define inverse bijections, and
thus we have the following result.
\begin{lemma}\label{lemma:mate-pairs}
  If $(f_0, g_0, \eta_0, \epz_0)$ and $(f_1, g_1, \eta_1, \epz_1)$ is a pair of adjunctions
  in $\B$, with $f_0\cn X_0 \to Y_0$ and $f_1\cn X_1 \to
  Y_1$, then taking mates establishes a bijection of 2-cells
  \[
  \B(X_0,Y_1)(f_1 a, b f_0) \iso \B(Y_0,X_1)(a g_0, g_1 b).
  \]
  for any 1-cells $a\cn X_0 \to X_1$ and $b\cn Y_0 \to Y_1$.
\end{lemma}

\begin{definition}\label{definition:internal-equivalence}
  An adjunction $(f,g,\eta,\epz)$ with $f\cn X \to Y$ and
  $g\cn Y \to X$ is called an \emph{internal equivalence} or
  \emph{adjoint equivalence} if $\eta$ and $\epz$ are isomorphisms.

  We say that $f$ and $g$ are members of an adjoint equivalence in
  this case, and we write $X \hty Y$ if such an equivalence exists.
  If $f$ is a member of an adjoint equivalence, we let $f^\bdot$
  denote an adjoint.
\end{definition}

Since mates are formed by pasting with unit/counit and unitors, we
have the following.
\begin{lemma}\label{lemma:mate-iso}
  If $(f,f^\bdot)$ is an adjoint equivalence, then a 2-cell
  $\theta\cn fs \to t$ is an isomorphism if and only if its mate
  $\theta^\dagger\cn s \to gt$ is an isomorphism.
\end{lemma}

\begin{definition}\label{definition:invertible-1-cell}
  A 1-cell $f\cn X \to Y$ is said to be \emph{invertible} or \emph{an
    equivalence} if there exists a 1-cell $g\cn Y \to X$ together with
  isomorphisms $gf \iso 1_X$ and $1_Y \iso fg$.
\end{definition}
Clearly the 1-cells in an adjoint equivalence are invertible.  The
converse also holds, by a standard argument modifying one of the two
isomorphisms in \cref{definition:invertible-1-cell}.
\begin{proposition}\label{proposition:equiv-via-isos}
  A 1-cell $f\cn X \to Y$ in $\sB$ is an equivalence if and only if it
  is a member of an adjoint equivalence.  
\end{proposition}

\subsection{Biequivalences}
Now we give the definitions of invertible strong transformation and biequivalence.

\begin{definition}
  Suppose that $F$ and $G$ are pseudofunctors of bicategories $\B \to
  \C$ and $\al\cn F \to G$ a strong transformation.  We say that $\al$
  is \emph{invertible} and write
  \[
  \al\cn F \fto{\hty} G
  \]
  if there is a strong transformation $\al^\bdot\cn G \to F$ together
  with invertible modifications
  \[
  \Theta\cn 1_F \iso \al^\bdot \al \quad \text{and} \quad \Gamma\cn
  \al \al^\bdot \iso 1_G.
  \]
\end{definition}

\begin{remark}
  The invertible strong transformations are the invertible 1-cells in
  a bicategory of pseudofunctors, strong transformations, and
  modifications.  However, we will not need this infrastructure.
  Instead, we will make use of the following characterization
  result.\dqed
\end{remark}

\begin{proposition}\label{proposition:adjoint-equivalence-componentwise}
  Suppose that $F$ and $G$ are pseudofunctors of bicategories $\B \to
  \C$ and suppose that $\al\cn F \to G$ is a strong transformation.  Then $\al$ is
  invertible if and only if each $\al_X\cn F(X) \to G(X)$ is an
  invertible 1-cell in $\C$.
\end{proposition}
\begin{proof}
  One implication is immediate.  For the other, suppose that $\al$ is
  a strong transformation and each component $\al_X$ is invertible.
  By \cref{proposition:equiv-via-isos} we may choose an adjoint
  inverse $\al^\bdot_X$ for each component.

  We will show that these components assemble to give a strong
  transformation $\al^\bdot\cn G \to F$ together with invertible
  modifications $\eta\cn 1_F \iso \al^\bdot \al$ and $\epz\cn \al
  \al^\bdot \iso 1_G$. We define the 2-cell aspect of $\al^\bdot$ by
  taking component-wise mates of the 2-cells for $\al$.  The
  transformation axioms for $\al^\bdot$ follow from those of $\al$ by
  \cref{lemma:mate-pairs}.  Each mate of an isomorphism is again an
  isomorphism by \cref{lemma:mate-iso}, and therefore $\al^\bdot$ is a
  strong transformation.  The componentwise units and counits
  define the requisite invertible modifications to make $\al$ and
  $\al^\bdot$ invertible strong transformations.
\end{proof}

\begin{definition}\label{definition:biequivalence}
  A pseudofunctor $F\cn \B \to \C$ is a \emph{biequivalence} if
  there exists a pseudofunctor $G\cn \C \to \B$ together with
  invertible strong transformations
  \[
  \Id_\B \fto{\hty} GF \quad \mathrm{ and } \quad FG \fto{\hty} \Id_\C.
  \]
\end{definition}

\begin{definition}\label{definition:es-ef-ff}
  Suppose $F\cn \B \to \C$ is a lax functor of bicategories.
  \begin{itemize}
  \item We say that $F$ is \emph{essentially surjective} if it is
    surjective on adjoint-equivalence classes of objects.
  \item We say that $F$ is \emph{essentially full} if it is surjective
    on isomorphism classes of 1-cells.
  \item We say that $F$ is \emph{fully faithful} if it is a bijection
    on 2-cells.
  \end{itemize}
\end{definition}

\begin{lemma}\label{lemma:biequiv-implies-local-equiv}
  If $F$ is a biequivalence, then each local functor
  \[
  \B(X,Y) \to \C(FX,FY)
  \]
  is essentially surjective and fully faithful.  That is, $F$ is
  essentially full on 1-cells and fully faithful on 2-cells.
\end{lemma}
\begin{proof}
  If $F$ and $G$ are inverse biequivalences, then one has local
  equivalences of categories
  \[
  \B(X,Y) \fto{\hty} \B((GF)X, (GF)Y) \quad \text{and} \quad
  \C(Z,W) \fto{\hty} \C((FG)Z, (FG)W)
  \]
  for $X$, $Y$ in $\B$ and $Z$, $W$ in $\C$.  The composites $GF$ and $FG$
  are fully faithful on 2-cells, and therefore by factoring the
  equivalences above one concludes that $F$ and
  $G$ are fully faithful on 2-cells.  Moreover, given any 1-cell
  $h \cn FX \to FY$, the composite $GF$ is essentially full on 1-cells and
  therefore there is some $\overline{h}\cn X \to
  Y$ such that $(GF)\overline{h} \iso Gh$.  But since $G$ is fully faithful, this implies that
  $F\overline{h} \iso h$ in $\B(FX,FY)$.  Thus $F$ is essentially
  full on 1-cells.
\end{proof}
\begin{remark}
  By the Whitehead Theorem for 1-categories,
  \cref{lemma:biequiv-implies-local-equiv} implies that a
  biequivalence is a local equivalence of categories, but we will not
  make use of this conclusion.\dqed
\end{remark}

%% file: parts/whitehead.tex

\section{The Lax Slice Bicategory}\label{sec:lax-slice}

In this section we describe a bicategorical generalization of slice
categories.

\begin{definition}
  Given a lax functor $F \cn \B \to \C$ and an object $X \in \C$, the
  \emph{lax slice} bicategory $F \dar X$ consists of the following.
  \begin{enumerate}
  \item Objects are pairs $(A,f_A)$ where $A \in \B$ and $FA \fto{f_A} X$ in
    $\C$.
  \item 1-cells $(A_0, f_0) \to (A_1,f_1)$ are 
    pairs $(p,\theta_p)$ where $A_0 \fto{p} A_1$ in $\B$ and
    $\theta_p\cn f_{0} \to f_{1} (Fp)$ in $\C$.  We depict this as a triangle.
    \[
      \begin{tikzpicture}[x=25mm,y=25mm]
        \draw[0cell] 
        (0,0) node (x) {X}
        (120:1) node (0) {FA_0}
        (60:1) node (1) {FA_1}
        ;
        \draw[1cell] 
        (0) edge[swap] node {f_0} (x)
        (1) edge node {f_1} (x)
        (0) edge node {Fp} (1) 
        ;
        \draw[2cell] 
        (-.05,.6) node[rotate=30,font=\Large] {\Rightarrow}
        node[below right] {\theta_p}
        ;
      \end{tikzpicture}
    \]
  \item 2-cells $(p_0,\theta_0) \to (p_1,\theta_1)$ are
    singletons $(\al)$ where $\al$ is a 2-cell $p_0 \to p_1$ in $\B$ such that $F\al$
    satisfies the equality shown in the pasting diagram below, known as
    the \emph{ice cream cone condition} with respect to $\theta_0$
    and $\theta_1$.
    \[
      \begin{tikzpicture}[x=35mm,y=35mm, scale=.85]
        \draw (0,0) node[font=\Large] {=};
        \newcommand{\boundary}{
          \draw[0cell] 
          (0,0) node (x) {X}
          (120:1) node (0) {FA_0}
          (60:1) node (1) {FA_1}
          ;
          \draw[1cell] 
          (0) edge[swap] node {f_0} (x)
          (1) edge node {f_1} (x)
          (0) edge[bend left] node (F1) {Fp_1} (1) 
          ;
        }
        \begin{scope}[shift={(-.8,-.5)}]
          \boundary
          \draw[1cell] 
          (0) edge[swap, bend right] node (F0) {Fp_0} (1) 
          ;
          \draw[2cell] 
          (0,.38) node[rotate=30,font=\Large] (A) {\Rightarrow}
          (A) ++(.05,-.1) node {\theta_0}
          node[between=F0 and F1 at .5, shift={(-.1,0)}, rotate=90, font=\Large] (Fal) {\Rightarrow}
          (Fal) node[right] {F\al}
          ;
        \end{scope}
        \begin{scope}[shift={(.8,-.5)}]
          \boundary
          \draw[1cell] 
          ;
          \draw[2cell] 
          (0,.6) node[rotate=30,font=\Large] (B) {\Rightarrow}
          (B) ++(.05,-.1) node {\theta_1}
          ;
        \end{scope}
      \end{tikzpicture}
    \]
  \end{enumerate}
  We describe the additional data of $F \dar X$ and prove that it
  satisfies the bicategory axioms in \cref{defprop:lax-slice}.
\end{definition}

\begin{proposition}\label{defprop:lax-slice}
  Given a lax functor $F\cn \B \to \C$ and an object $X \in \C$, the
  lax slice $F \dar X$ is a bicategory.
\end{proposition}
\begin{proof}
  The objects, 1-cells, and 2-cells of $F \dar X$ are defined above.
  We structure the rest of the proof as follows:
  \begin{enumerate}
  \item\label{it:slice-1} define identity 1-cells and 2-cells;
  \item\label{it:slice-2} define horizontal and vertical composition for 1-cells and
    2-cells;
  \item\label{it:slice-3} verify each collection of 1-cells and 2-cells between a given pair of
    objects forms a category;
  \item\label{it:slice-4} verify functoriality of horizontal composition;
  \item\label{it:slice-5} define components of the associator and unitor;
  \item\label{it:slice-6} verify that the associator and unitors are natural
    isomorphisms; and
  \item\label{it:slice-7} verify the pentagon and unity axioms.
  \end{enumerate}
  \newcommand{\step}[1]{\textbf{Step (\ref{it:slice-#1}).}}

  \step{1} The identity 1-cell for an object $(A,f_A)$ is $(1_A,r')$
  where
  \[
  r'= (1_{f_A} * F^0) \circ r^{-1},
  \]
  shown in the pasting diagram below.
  \begin{equation}\label{slice-1-1}
    \begin{tikzpicture}[x=35mm,y=35mm,baseline=(B.base),scale=.85]
      \draw[0cell] 
      (0,0) node (x) {X}
      (120:1) node (a) {FA}
      (60:1) node (b) {FA}
      ;
      \draw[1cell] 
      (a) edge[swap] node {f_{A}} (x)
      (b) edge node {f_{A}} (x) 
      (a) edge[bend left] node (T) {F1_{A}} (b) 
      (a) edge[swap, bend right] node (B) {1_{FA}} (b) 
      ;
      \draw[2cell] 
      node[between=x and B at .6, rotate=30, font=\Large] (A) {\Rightarrow}
      (A) ++(.02,0) node[below] {r^\inv}
      node[between=a and b at .5, rotate=90, font=\Large] (C) {\Rightarrow}
      (C) node[right] {F^0_{A}}
      ;
    \end{tikzpicture}
  \end{equation}
  The identity 2-cell for a 1-cell $(p,\theta)$ is given by $(1_p)$,
  noting that this satisfies the necessary condition because $F1_p =
  1_{Fp}$.
  
  \step{2} The horizontal composite of 1-cells
  \[
    (A_0,f_0) \fto{(p_0,\theta_0)}
    (A_1,f_1) \fto{(p_1,\theta_1)}
    (A_2,f_2)
  \]
  is $(p_1p_0, {\theta'})$, where ${\theta'}$ is given by the
  composite of the
  pasting diagram formed from $\theta_0$, $\theta_1$, and $F^2$ as shown below.
  \begin{equation}\label{slice-2-1}
    \begin{tikzpicture}[x=15mm,y=20mm,baseline=(f1.base),scale=.85]
      \draw[0cell] 
      (0,0) node (x) {X}
      (-2,2) node (a) {FA_0}
      (0,1.75) node (b) {FA_1}
      (2,2) node (c) {FA_2}
      ;
      \draw[1cell] 
      (a) edge[swap] node (f0) {f_{0}} (x)
      (b) edge node (f1) {f_{1}} (x) 
      (c) edge node (f2) {f_{2}} (x) 
      (a) edge[bend right=10] node (b0) {Fp_0} (b) 
      (b) edge[bend right=10] node (b1) {Fp_1} (c)
      (a) edge[bend left] node (T) {F(p_1p_0)} (c)
      ;
      \draw[2cell] 
      node[between=f0 and b at .5, rotate=45, font=\Large] (A) {\Rightarrow}
      (A) node[below right] {\theta_0}
      node[between=f2 and b at .5, rotate=45, font=\Large] (B) {\Rightarrow}
      (B) node[below right] {\theta_1}
      node[between=b and T at .5, rotate=90, font=\Large] (C) {\Rightarrow}
      (C) node[right] {F^2_{p_1, p_0}}
      ;
    \end{tikzpicture}
  \end{equation}

  Horizontal and vertical composites of 2-cells in $F \dar X$ are
  given by their composites in $\B$, as we now explain.
  Given 1-cells and 2-cells
  \begin{equation}\label{slice-2-2}
    \begin{tikzpicture}[x=40mm,y=20mm,baseline=(a0.base),scale=.8]
      \draw[0cell] 
      (0,0) node (a0) {(A_0,f_0)}
      (1,0) node (a1) {(A_1,f_1)}
      (2,0) node (a2) {(A_2,f_2)}
      ;
      \draw[1cell] 
      (a0) edge[bend right, swap] node {(p_0,\theta_0)} (a1) 
      (a1) edge[bend right, swap] node {(p_1,\theta_1)} (a2) 
      (a0) edge[bend left] node (X) {(p_0',\theta_0')} (a1) 
      (a1) edge[bend left] node {(p_1',\theta_1')} (a2) 
      ;
      \draw[2cell] 
      node[between=a0 and a1 at .5, rotate=90, font=\Large] (al0) {\Rightarrow}
      (al0) node[right] {(\al_0)}
      node[between=a1 and a2 at .5, rotate=90, font=\Large] (al1) {\Rightarrow}
      (al1) node[right] {(\al_1)}
      ;
    \end{tikzpicture}
  \end{equation}
  the following equalities of pasting diagrams show that $F(\al_1 *
  \al_0)$ satisfies the necessary condition for $\al_1 * \al_0$ to define a 2-cell in $F
  \dar X$.
  The first equality follows by naturality of $F^2$. The second follows by the conditions for
  $(\al_0)$ and $(\al_1)$ separately.
  \begin{equation*}
    \begin{tikzpicture}[x=11mm,y=18mm,baseline={(A0.base)}, scale=.97]
      \def\w{3.5} 
      \def\h{-2.3} 
      \def\m{.6} 
      \newcommand{\boundary}{
        \draw[0cell] 
        (0,0) node (x) {X}
        (-2,2) node (a) {FA_0}
        (0,1.5) node (b) {FA_1}
        (2,2) node (c) {FA_2}
        ;
        \draw[1cell] 
        (a) edge[swap] node (f0) {f_{0}} (x)
        (b) edge node (f1) {f_{1}} (x) 
        (c) edge node (f2) {f_{2}} (x) 
        (a) edge[bend left=60, looseness=1.5] node (T') {F(p_1'p_0')} (c)
        ;
        \draw[2cell] 
        ;
      }
      \begin{scope}[shift={(0,0)}]
        \boundary
        \draw[1cell] 
        (a) edge[bend left=10] node (T) {F(p_1p_0)} (c) 
        (a) edge[bend right=20] node (b0) {} (b) 
        (b) edge[bend right=20] node (b1) {} (c)
        ;
        \draw[2cell] 
        node[between=T and T' at .4, shift={(-.4,0)}, rotate=90, font=\Large] (A) {\Rightarrow}
        (A) node[right] {F(\al_1 * \al_0)} 
        node[between=b and T at .5, rotate=90, font=\Large] (C) {\Rightarrow}
        (C) node[right] {F^2_{p_1, p_0}}
        node[between=f0 and b at .5, rotate=45, font=\Large] (A) {\Rightarrow}
        (A) node[below right] {\theta_0}
        node[between=f2 and b at .55, rotate=45, font=\Large] (B) {\Rightarrow}
        (B) node[below right] {\theta_1}
        ;
      \end{scope}
      \begin{scope}[shift={(\w,\h)}]
        \boundary
        \draw[font=\LARGE] (x) ++(130:3.4) node[rotate=-45] (eq) {=};
        \draw[font=\LARGE] (x) ++(50:3.4) node[rotate=45] (eq) {=};
        \draw[1cell] 
        (a) edge[bend left=45, looseness=1.25] node {} (b) 
        (b) edge[bend left=45, looseness=1.25] node {} (c) 
        (a) edge[bend right=20] node (b0) {} (b) 
        (b) edge[bend right=20] node (b1) {} (c)
        ;
        \draw[2cell] 
        node[between=b and T' at .6, shift={(-.4,0)}, rotate=90, font=\Large] (A0) {\Rightarrow}
        (A0) node[right] {F^2_{p_1',p_2'}}
        node[between=a and b at .5, shift={(-.2,.1)}, rotate=90, font=\Large] (Fal0) {\Rightarrow}
        (Fal0) node[right] {F\al_0} 
        node[between=b and c at .5, shift={(-.25,.15)}, rotate=90, font=\Large] (Fal1) {\Rightarrow}
        (Fal1) node[right] {F\al_1} 
        node[between=f0 and b at .5, rotate=45, font=\Large] (A) {\Rightarrow}
        (A) node[below right] {\theta_0}
        node[between=f2 and b at .55, rotate=45, font=\Large] (B) {\Rightarrow}
        (B) node[below right] {\theta_1}
        ;
      \end{scope}
      \begin{scope}[shift={(\w+\w,0)}]
        \boundary
        \draw[1cell] 
        (a) edge[bend left=45, looseness=1.25] node {} (b) 
        (b) edge[bend left=45, looseness=1.25] node {} (c) 
        ;
        \draw[2cell] 
        node[between=b and T' at .6, shift={(-.4,0)}, rotate=90, font=\Large] (A2) {\Rightarrow}
        (A2) node[right] {F^2_{p_1',p_2'}}
        node[between=f0 and b at .5, shift={(135:.25)}, rotate=45, font=\Large] (A3) {\Rightarrow}
        (A3) node[below right] {\theta_0'}
        node[between=f2 and b at .55, shift={(45:.25)}, rotate=45, font=\Large] (B2) {\Rightarrow}
        (B2) node[below right] {\theta_1'}
        ;
      \end{scope}
    \end{tikzpicture}
  \end{equation*}
  Likewise, given $\al$ and $\al'$ as below,  
  \begin{equation}\label{slice-2-4}
    \begin{tikzpicture}[x=45mm,y=20mm,baseline=(a0.base),scale=.85]
      \draw[0cell] 
      (0,0) node (a0) {(A_0,f_0)}
      (1,0) node (a1) {(A_1,f_1)}
      ;
      \draw[1cell] 
      (a0) edge[bend right=60, swap, looseness=1.5] node (p) {(p,\theta)} (a1) 
      (a0) edge node (p') {(p',\theta')} (a1) 
      (a0) edge[bend left=60, looseness=1.5] node (p'') {(p'',\theta'')} (a1) 
      ;
      \draw[2cell] 
      node[between=p and p' at .5, rotate=90, font=\Large] (al0) {\Rightarrow}
      (al0) node[right] {(\al)}
      node[between=p' and p'' at .5, rotate=90, font=\Large] (al1) {\Rightarrow}
      (al1) node[right] {(\al')}
      ;
    \end{tikzpicture}
  \end{equation}
  the composite $\al'\, \al$ satisfies the necessary condition to
  define a 2-cell
  \[
    (\al'\,\al)\cn (p,\theta) \to (p'', \theta'')
  \]
  because $F$ is functorial with respect to composition of 2-cells.

  \step{3} Vertical composition in $F\dar X$ is strictly associative
  and unital because it is defined $\B$.  Therefore each collection of
  1-cells and 2-cells between a given pair of objects forms a
  category.

  \step{4} Likewise, because horizontal composition of 2-cells in
  $F\dar X$ is defined by the horizontal composites in $\B$, and these
  are functorial, it follows that horizontal composition of 2-cells in
  $F \dar X$ is functorial.

  \step{5} The remaining data to describe in $F \dar X$ are the associator and
  two unitors.  Consider a composable triple of 1-cells
    \[
    \begin{tikzpicture}[x=30mm,y=20mm]
      \draw[0cell] 
      (0,0) node (a0) {(A_0,f_0)}
      (1,0) node (a1) {(A_1,f_1)}
      (2,0) node (a2) {(A_2,f_2)}
      (3,0) node (a3) {(A_3,f_3).}
      ;
      \draw[1cell] 
      (a0) edge node {(p_0,\theta_0)} (a1) 
      (a1) edge node {(p_1,\theta_1)} (a2) 
      (a2) edge node {(p_2,\theta_2)} (a3) 
      ;
    \end{tikzpicture}
  \]
  Lax associativity \eqref{f2-bicat} for $F$ gives an equality of
  pasting diagrams shown below.
  \begin{equation}\label{slice-5-1}
    \begin{tikzpicture}[x=20mm,y=20mm,baseline={(0,1)},scale=.75]
      \draw[font=\Large] (2.55,.5) node (eq) {=}; 
      \newcommand{\boundary}{
        \draw[0cell] 
        (0,0) node (0) {FA_0}
        (3,0) node (3) {FA_3}
        (1.25,-1) node (4) {FA_1}
        ;
        \draw[1cell] 
        (0) edge[swap] node (c0) {Fp_0} (4)
        (4) edge[swap] node (21) (p21) {(Fp_2) \circ (Fp_1)} (3)
        (0) edge[bend left=85,looseness=1.9] node (c210R) {F(p_2(p_1p_0))} (3) 
        ;
        \draw[2cell] 
        ;
      }
      \begin{scope}[shift={(3,0)}]
        \boundary
        \draw[0cell] 
        (2,0) node (2) {FA_2}
        ;
        \draw[1cell] 
        (0) edge[swap] node[scale=.9] (p10) {(Fp_1) \circ (Fp_0)} (2) 
        (2) edge node (c2) {F(p_2)} (3)
        (0) edge[bend left=45,looseness=1.25] node[pos=.5] (c10) {F(p_1p_0)} (2)
        ;
        \draw[2cell] 
        (4) ++(0,.5) node[rotate=90, font=\Large] {\Rightarrow}
        node[right] {a_\C}
        node[between=p10 and c10 at .5, rotate=90, font=\Large] (F10) {\Rightarrow}
        (F10) node[right] {F^2}
        node[between=2 and c210R at .5, rotate=90, font=\Large] (F210R) {\Rightarrow}
        (F210R) node[right] {F^2}
        ;
      \end{scope}
      \begin{scope}[shift={(-1,0)}]
        \boundary
        \draw[0cell] 
        ;
        \draw[1cell] 
        (4) edge[bend left=40, looseness=1.1] node[pos=.8] (c21) {F(p_2p_1)} (3)
        (0) edge[bend left=50,looseness=1.15] node (c210L) {F((p_2p_1)p_0)} (3) 
        ;
        \draw[2cell] 
        node[between=p21 and c21 at .5, shift={(180:.15)}, rotate=120, font=\Large] (F21) {\Rightarrow}
        (F21) ++(210:.2) node {F^2}
        node[between=c0 and c210L at .5, shift={(0:.1)}, rotate=90, font=\Large] (F210L) {\Rightarrow}
        (F210L) node[left] {F^2}
        node[between=c210L and c210R at .5, rotate=90, font=\Large] (Fa) {\Rightarrow}
        (Fa) node[right] {Fa_\B}
        ;
      \end{scope}
    \end{tikzpicture}
  \end{equation}
  Combining these with the triangles
  \begin{equation}\label{slice-5-2}
    \begin{tikzpicture}[x=24mm,y=20mm,baseline=(f0.base)]
      \draw[0cell] 
      (0,0) node (0) {FA_0}
      (1,0) node (1) {FA_1}
      (2,0) node (2) {FA_2}
      (3,0) node (3) {FA_3}
      (1.5,-1.5) node (x) {X}
      ;
      \draw[1cell] 
      (0) edge node {Fp_0} (1) 
      (1) edge node {Fp_1} (2) 
      (2) edge node {Fp_2} (3)
      (0) edge[swap] node[pos=.33] (f0) {f_0} (x) 
      (1) edge[swap] node[pos=.33] (f1) {f_1} (x) 
      (2) edge node[pos=.33] (f2) {f_2} (x) 
      (3) edge node[pos=.33] (f3) {f_3} (x) 
      ;
      \draw[2cell] 
      node[between=f0 and 1 at .5, shift={(0,-.1)}, rotate=45, font=\Large] (th0) {\Rightarrow}
      (th0) node[above left] {\theta_0}
      node[between=f1 and 2 at .5, shift={(0,-.1)}, rotate=45, font=\Large] (th1) {\Rightarrow}
      (th1) node[above left] {\theta_1}
      node[between=f3 and 2 at .5, shift={(0,-.1)}, rotate=45, font=\Large] (th2) {\Rightarrow}
      (th2) node[above left] {\theta_2}
      ;
    \end{tikzpicture}
  \end{equation}
  shows that $Fa_\B$ satisfies the relevant ice cream cone condition and hence
  $a_\B$ defines a 2-cell
  \[
    (a_\B)\cn ((p_2,\theta_2)(p_1,\theta_1))\, (p_0,\theta_0) \to
    (p_2,\theta_2) \, ((p_1,\theta_1) (p_0,\theta_0)).
  \]
  in $F\dar X$.  Note that one must implicitly make use of associators to
  interpret pasting diagrams of three triangles; the component of 
  $a_\C$ in \eqref{slice-5-1} cancels with its inverse to form the
  composite in the target of $(a_\B)$.

  The left and right unitors are defined similarly: the unitors $r_\B$ and $\ell_\B$
  satisfy the appropriate ice cream cone conditions and therefore given
  a 1-cell $(p,\theta)\cn (a_0,f_0) \to (a_1, f_1)$, we have 2-cells
  \[
    (r_\B)\cn (p,\theta) (1_{A_0},r') \to (p,\theta) \quad \mathrm{ and }
    \quad
    (\ell_\B) \cn (1_{A_1},r') (p,\theta) \to (p,\theta).
  \]

  \step{6} Naturality of the associator and unitors defined in the
  previous step is a consequence of the corresponding naturality in
  $\B$ and $\C$ together with naturality of $F^0$ and $F^2$.
  Moreover, each component is an isomorphism because a lax functor
  preserves invertibility of 2-cells.

  \step{7} Because the associator and unitor are defined by the
  corresponding components in $\B$, it follows that they satisfy the
  unity and pentagon axioms, \eqref{bicat-unity} and
  \eqref{bicat-pentagon}.
\end{proof}

\begin{proposition}\label{lemma:base-change-functor}
  Suppose $F\cn \B \to \C$ is a lax functor of bicategories.  Given a
  1-cell $u\cn X \to Y$, there is a strict functor
  \[
    F \dar u \cn (F \dar X) \to (F \dar Y)
  \]
  induced by whiskering with $u$.
\end{proposition}
\begin{proof}
  The assignment on 0-, 1- and 2-cells, respectively, is given by
  \begin{align*}
    (A,f_A) & \mapsto (A, uf_A)\\
    (p,\theta) & \mapsto (p, a_\C^\inv \circ (1_u * \theta)) \\
    (\al) & \mapsto (\al).
  \end{align*}
  where the associator $a_\C$ is used to ensure that the target of the
  2-cell $a_\C^\inv \circ (1_u * \theta)$ is $(u f_{A_1}) \circ (Fp)$.

  To show that $F \dar u$ is strictly unital, recall that the identity
  1-cell of $(A,f_A)$ is $(1_A,r')$ where
  \[
  r' = (1_{f_A} * F^0) \circ r^\inv
  \]
  is shown in \eqref{slice-1-1}.  Then, using the functoriality of
  $(1_u * -)$, the 2-cell
  component of $(F\dar u)(1_A,r')$ is shown along the top and right of
  the diagram below. The right unity property from
  \cref{bicat-left-right-unity} together with naturality of $a_\C$
  shows that the diagram commutes and therefore $F\dar u$ is strictly
  unital.
  \[
  \begin{tikzpicture}[x=30mm,y=20mm]
    \draw[0cell] 
    (0,0) node (a) {uf_A}
    (1,0) node (b) {u(f_A 1_{FA})}
    (2.5,0) node (c) {u(f_A F1_A)}
    (1,-1) node (d) {(uf_A)1_{FA}}
    (2.5,-1) node (e) {(uf_A) F1_A}
    ;
    \path[1cell] 
    (a) edge node {1_u * r^\inv} (b)
    (b) edge node {1_u * (1_{f_A} * F^0)} (c)
    (c) edge node {a^\inv_\C} (e)
    (b) edge node {a^\inv_\C} (d)
    (d) edge[swap] node {1_{uf_A} * F^0} (e)
    (a) edge[swap] node {r^\inv} (d)
    ;
    \draw[2cell] 
    
    ;
  \end{tikzpicture}
  \]
  A similar calculation using the functoriality of whiskering and
  naturality of the associator shows that $F\dar u$ is strictly
  functorial with respect to horizontal composition.
\end{proof}

\begin{definition}
  We call the strict functor $F \dar u$ constructed in
  \cref{lemma:base-change-functor} the \emph{change-of-slice functor}.
\end{definition}

\section{Lax Terminal Objects in Lax Slices}\label{sec:lax-terminal}

In this section we introduce a specialized notion of terminal object
called inc-lax terminal and prove two key results.  First,
\cref{proposition:lax-slice-lax-terminal} proves that if a lax functor
$F$ is essentially surjective, essentially full, and fully faithful,
then the lax slices can be equipped with our specialized form of
terminal object.  Second, \cref{lemma:lax-slice-change-fiber} proves
that if $F$ is furthermore a pseudofunctor, then these terminal
objects are preserved by change-of-slice functors.  These are the two
key properties of lax slices required for the construction of a
reverse lax functor in \cref{sec:Quillen-A-bicat}.

Given an object $X$ of a bicategory $\C$, recall
\cref{constant-pseudofunctor} describes $\conof{X}$, the constant
pseudofunctor at $X$.

\begin{definition}\label{definition:lax-terminal}
  We say that $\lto \in \C$ is \emph{lax terminal} if there is a
  lax transformation $k \cn \Id_\C \to \conof{\lto}$.  Such a
  transformation has component 1-cells $k_X\cn X \to \lto$ for $X
  \in \C$ and 2-cells
  \[
  \begin{tikzpicture}[x=16mm,y=16mm]
    \draw[0cell] 
    (0,0) node (1) {\lto}
    (1,0) node (1') {\lto}
    (1') ++(90:1) node (y) {Y}
    (1) ++(90:1) node (x) {X}
    ;
    \draw[1cell] 
    (x) edge node {u} (y)
    (x) edge[swap] node {k_X} (1)
    (y) edge node {k_Y} (1')
    (1) edge[swap] node {1_{\lto}} (1')
    ;
    \draw[2cell] 
    (.55,.45) node[rotate=45,font=\Large] (R) {\Rightarrow}
    (R) node[above left] {k_u}
    ;
  \end{tikzpicture}
  \]
  satisfying the lax unity and lax naturality axioms.
\end{definition}


\begin{definition}\label{definition:inc-lax-terminal}
  Given lax functors $F,G\cn \B \to \C$, we say that a lax
  transformation $k\cn F \to G$ is \emph{inc-lax} or
  \emph{\underline{in}itial-\underline{c}omponent-lax} if each
  component
  \[
  k_X\cn FX \to GX
  \]
  is initial in the category $\C(FX,GX)$.
\end{definition}

\begin{definition}
  Suppose that $\lto \in \C$ is a lax terminal object with lax
  transformation $k\cn \Id_\C \to \conof{\lto}$.
  We say $\lto$ is an \emph{inc-lax terminal} object
  if $k$ is inc-lax and the component $k_\lto$ at $\lto$ is
  the identity 1-cell $1_\lto$.
\end{definition}

\begin{explanation}
  The universal property of initial 1-cells implies that, for a 1-cell
  $u\cn X \to Y$, the lax naturality constraint $k_u$ is equal to the
  composite of the left unitor with the universal 2-cell from each
  $k_X$ to the composite $k_Y \, u$, as shown below.
  \[
  \begin{tikzpicture}[x=16mm,y=16mm]
    \newcommand{\boundary}{
      \draw[0cell] 
      (0,0) node (1) {\lto}
      (1,0) node (1') {\lto}
      (1') ++(90:1) node (y) {Y}
      (1) ++(90:1) node (x) {X}
      ;
      \draw[1cell] 
      (x) edge node {u} (y)
      (x) edge[swap] node {k_X} (1)
      (y) edge node {k_Y} (1')
      (1) edge[swap] node {1_{\lto}} (1')
      ;
    }
    \draw (1.75,.5) node[font=\Large] {=};
    \begin{scope}
      \boundary
      \draw[2cell] 
      (.55,.45) node[rotate=45,font=\Large] (R) {\Rightarrow}
      (R) node[above left] {k_u}
      ;
    \end{scope}
    \begin{scope}[shift={(2.5,0)}]
      \boundary
      \draw[1cell] 
      (x) edge[swap] node[pos=.25,inner sep=0] {k_X} (1')
      ;
      \draw[2cell] 
      (.4,.3) node[rotate=45,font=\Large] (R) {\Rightarrow}
      (R) node[below right] {\ell}
      (.75,.65) node[rotate=45,font=\Large] (E) {\Rightarrow}
      (E) ++(.075,-.16) node {\exists !}
      ;
    \end{scope}
  \end{tikzpicture}
  \]
\end{explanation}

\begin{definition}\label{definition:preserves-inc}
  Suppose that $\B$ and $\C$ have inc-lax terminal objects $(\lto,k)$
  and $(\lto',k')$, respectively.
  We say that a lax functor $F \cn \B \to \C$ \emph{preserves initial
    components} if each composite
  \[
  FX \fto{Fk_X} F\lto \fto{k'_{(F\lto)}} \lto'
  \]
  is initial in $\C(FX,\lto)$.
\end{definition}

\begin{lemma}\label{lemma:preserves-initial-1-cells}
  Suppose that $F\cn \B \to \C$ preserves initial components.  If
  \[
  f\cn X \to \lto
  \]
  is any initial 1-cell in
  $\B(X,\lto)$, then the composite
  \[
  FX \fto{Ff} F\lto \fto{k'_{(F\lto)}} \lto'
  \]
  is initial in $\C(FX,\lto')$.
\end{lemma}
\begin{proof}
  If $f$ is initial, then there is a unique isomorphism $f \iso
  k_X$.  Therefore $Ff \iso Fk_X$ and hence their composites
  with $k'_{(F\lto)}$ are isomorphic.  Now
  \[
  (k'_{(F\lto)}) \circ (Fk_X)
  \]
  is inital by hypothesis, and therefore the result follows.
\end{proof}

Now we show that, if $F$ is essentially surjective, essentially full,
and fully faithful, then each lax slice $F\dar X$ has an inc-lax
terminal object, and each change-of-slice functor $F \dar u$ preserves initial
components.  The first of these results requires the axiom of choice,
and the second depends on the first.

\begin{proposition}\label{proposition:lax-slice-lax-terminal}
  Suppose $F$ is a lax functor which is essentially surjective,
  essentially full, and fully faithful.  Then for each $X \in \C$ the
  lax slice $F \dar X$ has an inc-lax terminal object.
\end{proposition}
\begin{proof}
  Since $F$ is essentially surjective on objects, there is a choice of
  object $\ol{X} \in \B$ and invertible 1-cell
  \[
  f_{\ol{X}}\cn F\ol{X} \to X
  \]
  with adjoint inverse
  \[
  f_{\ol{X}}^\bdot\cn X \to F\ol{X}.
  \]
  Therefore
  $(\ol{X},f_{\ol{X}})$ is an object of $F \dar X$; we will show that
  it is an inc-lax terminal object.  Given any other object $(A, f_A)$ in $F \dar X$, we
  have a composite
  \[
    FA \fto{f_A} X \fto{f_{\ol{X}}^{\bdot}} F\ol{X}
  \]
  in $\C$.  Since $F$ is essentially surjective on 1-cells, there is a
  choice of 1-cell $p_A$ together with a 2-cell isomorphism
  \[
  \theta^{\dagger}_A\cn f_{\ol{X}}^\bdot \; f_A \to Fp_A
  \]
  whose mate $\theta_A$ fills the triangle
\[
  \begin{tikzpicture}[x=20mm,y=20mm]
    \draw[0cell] 
    (0,0) node (x) {X}
    (120:1) node (0) {FA}
    (60:1) node (1) {F\ol{X}}
    ;
    \draw[1cell] 
    (0) edge[swap] node {f_A} (x)
    (1) edge node {f_{\ol{X}}} (x)
    (0) edge node {Fp_A} (1) 
    ;
    \draw[2cell] 
    (.05,.5) node[rotate=30,font=\Large] {\Rightarrow}
    node[above left] {\theta_A}
    ;
  \end{tikzpicture}
\]
Note that $\theta_A$ is therefore also an isomorphism
by \cref{lemma:mate-iso}.  If $(A,f_A)$ is equal to the object
$(\ol{X},f_{\ol{X}})$, then we require the choice of
$(p_{\ol{X}},\theta_{\ol{X}})$ to be the identity 1-cell $(1_{\ol{X}},r')$
described in \eqref{slice-1-1}.

Therefore $(p_A,\theta_A)$ defines a 1-cell $(A,f_A) \to
(\ol{X},f_{\ol{X}})$ in $F \dar X$ which is the identity 1-cell if
$(A,f_A) = (\ol{X},f_{\ol{X}})$.  Now we show that $(p_A,\theta_A)$ is
initial in the category of 1- and 2-cells $(A,f_A) \to
(X,f_{\ol{X}})$.  The universal property for initial 1-cells then
implies that the components defined by
$k_{(A,f_A)} = (p_A,\theta_A)$ assemble to form a lax
transformation to the constant pseudofunctor at $(\ol{X},f_{\ol{X}})$.

Given any other 1-cell $(q,\omega)\cn (A,f_A) \to
(\ol{X},f_{\ol{X}})$, we compose with $\theta_A^\inv$ to obtain a 2-cell
\[
{\ga'}\cn f_{\ol{X}} \; (Fp_A) \to f_{\ol{X}} \; (Fq)
\]
shown below.
\[
  \begin{tikzpicture}[x=25mm,y=25mm]
    \draw[0cell] 
    (0,0) node (x) {X}
    (120:1) node (0) {FA}
    (60:1) node (1) {F\ol{X}}
    (180:1) node (2) {F\ol{X}}
    ;
    \path[1cell] 
    (0) edge node[pos=.75] {f_A} (x)
    (1) edge node {f_{\ol{X}}} (x)
    (0) edge node (Fpa) {Fq} (1)
    (0) edge[swap] node {Fp_{A}} (2) 
    (2) edge[swap] node {f_{\ol{X}}} (x)
    ;
    \draw[2cell] 
    (.05,.6) node[rotate=30,font=\Large] {\Rightarrow}
    node[above left] {\om}
    (150:.5) node[rotate=30,font=\Large] {\Rightarrow}
    node[above left] {\theta_A^\inv}

    ;
  \end{tikzpicture}
\]

Since $f_{\ol{X}}$ is an adjoint equivalence, this uniquely determines
a 2-cell
\[
\ga\cn Fp_A \to Fq
\]
such that $1_{f_{\ol{X}}} * \ga = \om\, \tha_A^\inv$.
Therefore, because $F$ is fully
faithful on 2-cells, we have a unique 2-cell
\[
\ol{\ga} \cn p_A \to q
\]
such that $F\ol{\ga} = \ga$ and hence satisfies the ice cream cone condition shown
below.
\begin{equation*}
  \begin{tikzpicture}[x=30mm,y=35mm,baseline={(0,-1.25)},scale=.9]
    \def\w{1.7} 
    \begin{scope}[shift={(-\w,0)}]
    \draw[0cell] 
    (0,0) node (x) {X}
    (120:1) node (0) {FA}
    (60:1) node (1) {F\ol{X}}
    ;
    \draw[1cell] 
    (0) edge[swap] node {f_A} (x)
    (1) edge node {f_{\ol{X}}} (x)
    (0) edge[bend right,swap] node (Fpa) {Fp_A} (1)
    (0) edge[bend left] node (Fq) {Fq} (1)
    ;
    \draw[2cell] 
    (.05,.35) node[rotate=30,font=\Large] {\Rightarrow}
    node[above left] {\theta_A}
    node[between=Fpa and Fq at .5, shift={(.05,0)}, rotate=90, font=\Large] (E) {\Rightarrow}
    (E) node[left] {F\ol{\ga}}
    ;
    \end{scope}
    \begin{scope}[shift={(-.75,.2)},scale=.75]
      \draw[font=\LARGE]
      (1,0) ++(1.4,.5) node[rotate=0] {=}
      (1,0) ++(-1.4,.5) node[rotate=0] {=}
      ;
      \draw[0cell]
      (0,0) node (Fa) {FA}
      (1,0) node (Fxt) {F\ol{X}}
      (2,0) node (x) {X}
      (1,1.25) node (Fxt2) {F\ol{X}}
      ;
      \draw[1cell]
      (Fa) edge[bend right=60, looseness=1.15] node (faB) {f_A} (x) 

      (Fa) edge[swap] node (Fpa) {Fp_A} (Fxt)
      (Fxt) edge node (fxt) {f_{\ol{X}}} (x)

      (Fa) edge[bend left=60, looseness=1.25] node (fa) {f_A} (x) 
      (Fa) edge[bend left] node (Fq) {Fq} (Fxt2)
      (Fxt2) edge[bend left] node (fxt2) {f_{\ol{X}}} (x)
      ;
      \draw[2cell]
      node[between=faB and Fxt at .5, rotate=90, font=\Large] (Ta) {\Rightarrow}
      (Ta) node[left] {\theta_A}
      node[between=Fxt and fa at .5, rotate=90, font=\Large] (Tainv) {\Rightarrow}
      (Tainv) node[right] {\theta_A^\inv}
      node[between=fa and Fxt2 at .5, rotate=90, font=\Large] (Om) {\Rightarrow}
      (Om) node[right] {\omega}
      ;
    \end{scope}
    \begin{scope}[shift={(\w,0)}]
      \draw[0cell] 
      (0,0) node (x) {X}
      (120:1) node (0) {FA}
      (60:1) node (1) {F\ol{X}}
      ;
      \draw[1cell] 
      (0) edge[swap] node {f_A} (x)
      (1) edge node {f_{\ol{X}}} (x)
      (0) edge[bend left] node (Fq) {Fq} (1)
      ;
      \draw[2cell] 
      (.05,.55) node[rotate=30,font=\Large] {\Rightarrow}
      node[above left] {\om}
      ;
    \end{scope}
  \end{tikzpicture}
\end{equation*}
Therefore $(\ol{\ga})$ is a 2-cell in $F \dar X$ from
$(p_A,\theta_A)$ to $(q,\om)$.  The diagram above, together with
the invertibility of $\theta_A$ and the uniqueness of both $\ga$ and
$\ol{\ga}$ implies that $(\ol{\ga})$ is the unique such 2-cell in $F
\dar X$.
\end{proof}

\begin{proposition}\label{lemma:lax-slice-change-fiber}
  Suppose $F$ is a pseudofunctor which is essentially surjective,
  essentially full, and fully faithful.  Then for each 1-cell
  $u\cn X \to Y$ in $\C$, the strict functor $F \dar u$ preserves
  initial components.
\end{proposition}
\begin{proof}
  For $(A, f_A) \in F \dar X$, let $(p_A, \theta_A)$ denote the
  initial 1-cell from $(A,f_A)$ to the inc-lax terminal object
  \[
  (\ol{X},f_{\ol{X}}) \in F \dar X.
  \]
  Let $(\ol{u},\theta_{\ol{u}})$ denote the initial 1-cell from
  \[
  (F \dar u)(\ol{X}, f_{\ol{X}}) = (\ol{X},u f_{\ol{X}})
  \]
  to the inc-lax terminal object
  \[
  (\ol{Y},f_{\ol{Y}}) \in F \dar Y.
  \]
  We must show that the composite of $(\ol{u},\theta_{\ol{u}})$ with $(F \dar
  u)(p_A,\theta_A)$ is initial.  This composite is given by $(\ol{u}
  p_A, {\theta'})$, where ${\theta'}$ is the 2-cell determined by the
  pasting diagram below.
  \begin{equation}\label{base-change-thetabar}
    \begin{tikzpicture}[x=15mm,y=20mm,baseline=(xt.base), scale=.85]
      \draw[0cell] 
      (0,0) node (y) {Y}
      (-.8,.8) node (x) {X}
      (-2,2) node (a) {FA}
      (.15,1.75) node (xt) {F\ol{X}}
      (2,2) node (yt) {F\ol{Y}}
      ;
      \draw[1cell] 
      (a) edge[swap] node (f0) {f_{A}} (x)
      (xt) edge node (f1) {f_{\ol{X}}} (x)
      (x) edge[swap] node {u} (y)
      (yt) edge node (f2) {f_{\ol{Y}}} (y) 
      (a) edge[bend right=10] node (b0) {Fp_A} (xt) 
      (xt) edge[bend right=10] node[pos=.75] (b1) {F\ol{u}} (yt)
      (a) edge[bend left] node (T) {F(\ol{u}\,p_A)} (yt)
      ;
      \draw[2cell] 
      node[between=f0 and xt at .4, rotate=45, font=\Large] (A) {\Rightarrow}
      (A) node[below right] {\theta_A}
      node[between=f2 and xt at .55, rotate=45, font=\Large] (B) {\Rightarrow}
      (B) node[below right] {\theta_{\ol{u}}}
      node[between=xt and T at .5, shift={(-.1,0)}, rotate=90, font=\Large] (C) {\Rightarrow}
      (C) node[right] {F^2_{p_1, p_0}}
      ;
    \end{tikzpicture}
  \end{equation}
  
  The argument in \cref{proposition:lax-slice-lax-terminal} shows that
  $\theta_A$ and $\theta_{\ol{u}}$ are isomorphisms.  Since
  $F$ is a pseudofunctor by hypothesis, the 2-cells $F^2$
  are isomorphisms and hence ${\theta'}$ is an isomorphism.  Then, as
  in the proof of \cref{proposition:lax-slice-lax-terminal},
  composition with the inverse of ${\theta'}$ shows that
  $(\ol{u}\, p_A, {\theta'})$ is initial.
\end{proof}

\section{Quillen Theorem A for Bicategories}\label{sec:Quillen-A-bicat}

In this section we explain how to construct a reverse lax functor $G$.
We assume only that $F$ is lax functor, that its lax slices are
equipped with inc-lax terminal objects, and that these are preserved
by change-of-slice.  The end of \cref{sec:lax-terminal} explains how,
with the axiom of choice, one can choose such data when $F$ is an
essentially surjective, essentially full, and fully faithful
pseudofunctor.  However, if one has a constructive method for
obtaining these data in practice, then \cref{theorem:Quillen-A-bicat}
gives a construction of $G$ which does not depend on choice.  In
\cref{sec:Whitehead-bicat} we show that, under the hypotheses of the
Bicategorical Whitehead Theorem \ref{theorem:whitehead-bicat}, the $G$
constructed here is an inverse biequivalence for $F$.

\begin{theorem}[Bicategorical Quillen Theorem A]\label{theorem:Quillen-A-bicat}
  Suppose $F\cn \B \to \C$ is a lax functor of bicategories and
  suppose the following:
  \begin{enumerate}
  \item\label{QA-hypothesis-1} For each $X \in \C$, the lax slice
    bicategory $F\dar X$ has an inc-lax terminal object
    $(\ol{X},f_{\ol{X}})$.
    Let $k^X$ denote the inc-lax transformation $\Id_{F\dar X} \to
    \conof{{(\ol{X},f_{\ol{X}})}}$.
  \item\label{QA-hypothesis-2} For each $u\cn X \to Y$ in $\C$, the
    induced functor $F\dar u$
    preserves initial components (\cref{definition:preserves-inc}).
  \end{enumerate}
  Then there is a lax functor $G \cn \C \to \B$ together with
  lax transformations
  \[
  \eta\cn \Id_\B \to GF \quad \mathrm{\ and\ } \quad \epz\cn FG \to \Id_\C.
  \]
\end{theorem}

The proof is structured as follows:
\begin{enumerate}
\item \label{it:G-1} \cref{definition:G}: define the data for $G =
  (G,G^2,G^0)$:
  \begin{enumerate}
  \item \label{it:G-1a} define $G$ as an assignment on 0-, 1-, and 2-cells;
  \item \label{it:G-1b} define the components of $G^0$ and $G^2$
  \end{enumerate}
\item \label{it:G-2} \cref{proposition:G-lax}: Show that $G$ defines a lax functor:
  \begin{enumerate}
  \item \label{it:G-2a} show that $G$ is functorial with respect to 2-cells;
  \item \label{it:G-2b} show that $G^2$ and $G^0$ are natural with respect to 2-cells;
  \item \label{it:G-2c} verify the lax associativity axiom \eqref{f2-bicat}
  \item \label{it:G-2d} verify the left and right unity axioms \eqref{f0-bicat}.
  \end{enumerate}
\item \label{it:G-3} Establish the existence of $\eta$ and $\epz$:
  \begin{enumerate}
  \item \label{it:G-3a} define the components of $\eta$ and $\epz$;
  \item \label{it:G-3b} verify the 2-cell components of $\eta$ and $\epz$ are natural
    with respect to 2-cells;
  \item \label{it:G-3c} verify the unity axiom \eqref{unity-transformation-pasting} for $\eta$
    and $\epz$;
  \item \label{it:G-3d} verify the horizontal naturality axiom
    \eqref{2-cell-transformation-pasting} for $\eta$ and $\epz$.
  \end{enumerate}
\end{enumerate}
\newcommand{\Gstep}[1]{\textbf{Step (\ref{it:G-#1}).}}
\newcommand{\Gsteps}[2]{\textbf{Steps (\ref{it:G-#1}) and (\ref{it:G-#2}).}}

\begin{definition}\label{definition:G}
  Suppose $F\cn\B \to \C$ is a lax functor satisfying the
  assumptions of \cref{theorem:Quillen-A-bicat}.

  \Gstep{1a}   We define an assignment on cells $G\cn\C \to \B$ as
  follows.
  \begin{itemize}
  \item For each object $X$ in $\C$, the slice $F \dar X$ has an inc-lax terminal
    object $(\ol{X},f_{\ol{X}})$.  Define $GX = \ol{X}$.
  \item For each 1-cell $u\cn X \to Y$ in $\C$, we have $(\ol{X},
    uf_{\ol{X}}) \in F \dar Y$, and inc-lax terminal object
    $(\ol{Y},f_{\ol{Y}}) \in F \dar Y$.
    The component of $k^Y$ at $(\ol{X},uf_{\ol{X}})$ is an initial 1-cell
    \[
      (\ol{u},\theta_{\ol{u}}) \cn (\ol{X}, uf_{\ol{X}}) \to (\ol{Y}, f_{\ol{Y}}).
    \]
    Define $Gu = \ol{u}$.
  \item Given a 2-cell $\ga\cn u_0 \to u_1$ in $\C$, we have 1-cells in
    $F\dar Y$ given by $(\ol{u_0},\theta_{0})$ and
    $(\ol{u_1},\theta_{1})$, the components of $k^Y$.
    Pasting the latter of these with $\ga$ yields a 1-cell $(\ol{u_1},
    \theta_1(\ga * 1_{f_{\ol{X}}}))$ shown in the pasting diagram below.
    \begin{equation}\label{Gdef-3}
      \begin{tikzpicture}[x=15mm,y=15mm,baseline={(x.base)}]
        \draw[0cell] 
        (0,0) node (y) {Y}
        (120:2) node (a) {F\ol{X}}
        (120:1) node (x) {X}
        (60:2) node (b) {F\ol{Y}}
        ;
        \draw[1cell] 
        (a) edge[swap] node {f_{\ol{X}}} (x)
        (x) edge[bend left] node {u_1} (y) 
        (x) edge[swap, bend right=40] node {u_0} (y) 
        (b) edge node {f_{\ol{Y}}} (y)
        (a) edge node (T) {F\ol{u_1}} (b)
        ;
        \draw[2cell] 
         node[between=y and T at .65, rotate=30, font=\Large] (A) {\Rightarrow}
         (A) node[below right] {\theta_{1}}
         node[between=x and y at .55, rotate=30, font=\Large] (B) {\Rightarrow}
         (B) node[above left] {\ga}
        ;
      \end{tikzpicture}
    \end{equation}
    Since $(\ol{u_0},\theta_0)$ is initial by construction and
    \[
    (\ol{u_1},\theta_1(\ga * 1_{f_{\ol{X}}}))
    \]
    is another 1-cell in $F \dar Y$ with source
    $(\ol{X},u_0f_{\ol{X}})$ and target $(\ol{Y},f_{\ol{Y}})$, there
    is a unique 2-cell $(\ol{\ga})$ in $F \dar Y$ such that
    $F\ol{\ga}$ satisfies the ice cream cone condition shown below.
    \begin{equation}\label{Gdef-4}
    \begin{tikzpicture}[x=17mm,y=17mm,baseline={(eq.base)}]
      \def\w{1.75} 
      \def\h{-1.5} 
      \def\m{.5} 
      \draw[font=\Large] (\w/2+\m,0) node (eq) {=}; 
      \newcommand{\boundary}{
        \draw[0cell] 
        (0,-.25) node (y) {Y}
        (120:2) node (a) {F\ol{X}}
        node[between=a and y at .5] (x) {X}
        (60:2) node (b) {F\ol{Y}}
        ;
        \draw[1cell] 
        (a) edge[swap] node {f_{\ol{X}}} (x)
        (b) edge node (fy) {f_{\ol{Y}}} (y) 
        (x) edge[swap, bend right] node {u_0} (y)
        (a) edge[bend left] node (T) {F\ol{u_1}} (b) 
        ;
      }
      \begin{scope}[shift={(0,\h/2)}]
        \boundary
        \draw[1cell] 
        (a) edge[swap, bend right] node (B) {F\ol{u_0}} (b) 
        ;
        \draw[2cell] 
        node[between=y and B at .6, rotate=50, font=\Large] (D) {\Rightarrow}
        (D) node[left, shift={(-.05,.05)}] {\theta_{0}}
        node[between=a and b at .5, rotate=90, font=\Large] (C) {\Rightarrow}
        (C) node[right] {F\ol{\ga}}
        ;
      \end{scope}      
      \begin{scope}[shift={(\w+\m+\m,\h/2)}]
        \boundary
        \draw[1cell] 
        (x) edge[bend left] node {u_1} (y) 
        ;
        \draw[2cell] 
        node[between=y and T at .5, shift={(.05,.1)}, rotate=30, font=\Large] (BB) {\Rightarrow}
        (BB) node[below, shift={(.05,-.05)}] {\theta_{1}}
        node[between=x and y at .55, rotate=30, font=\Large] (CC) {\Rightarrow}
        (CC) node[above left] {\ga}
        ;
      \end{scope}
    \end{tikzpicture}
    \end{equation}
    Define $G\ga = \ol{\ga}$.
  \end{itemize}

  \Gstep{1b} Next we define the components of the lax constraints $G^0$ and $G^2$.
  \begin{itemize}
  \item Following the definition of $G$ for $Y = X$ and $u = 1_X$, we
    obtain a 1-cell
    \[
    G1_X = \ol{1_X}\cn \ol{X} \to \ol{X}
    \]
    together with $\theta_{\ol{1_X}}$ filling the triangle below.
    \begin{equation}\label{Gdef-7}
      \begin{tikzpicture}[x=15mm,y=13mm,baseline=(x.base)]
        \draw[0cell] 
        (0,0) node (y) {X.}
        (120:2) node (a) {F\ol{X}}
        (120:1) node (x) {X}
        (60:2) node (b) {F\ol{X}}
        ;
        \draw[1cell] 
        (a) edge[swap] node {f_{\ol{X}}} (x)
        (x) edge[swap] node {1_X} (y)
        (b) edge node {f_{\ol{X}}} (y)
        (a) edge node (T) {F\ol{1_X}} (b)
        ;
        \draw[2cell] 
        node[between=y and T at .65, rotate=30, font=\Large] (A) {\Rightarrow}
        (A) node[below right] {\theta_{\ol{1_X}}}
        ;
      \end{tikzpicture}
    \end{equation}
    Composing $\theta_{\ol{1_X}}$ with the left unitor $\ell$ we obtain
    a 1-cell in $F \dar X$
    \[
    (\ol{1_X}, \ell_{f_{\ol{X}}} \circ \theta_{\ol{1_X}}) \cn (\ol{X},
    f_{\ol{X}}) \to (\ol{X}, f_{\ol{X}}).
    \]
    By the unit condition for inc-lax terminal objects, the identity 1-cell for
    $(\ol{X},f_{\ol{X}})$ is initial and hence we have a unique 2-cell
    \[
    1_{GX} = 1_{\ol{X}} \to \ol{1_X} = G1_X
    \]
    satisfying the ice cream cone condition for
    \[
    (\ol{1_X}, \ell_{f_{\ol{X}}} \circ \theta_{\ol{1_X}}) \andspace
    (1_{\ol{X}}, r').
    \]
    We define $G^0_X$ to be this 2-cell.
    
  \item Given a pair of composable arrows $u\cn X \to Y$ and $v\cn Y \to Z$ in $\C$, we have
    initial 1-cells $(\ol{u},\theta_{\ol{u}})$ and $(\ol{v},
    \theta_{\ol{v}})$ shown below.
    \begin{equation}\label{Gdef-9}
    \begin{tikzpicture}[x=15mm,y=13mm,baseline={(x.base)}]
      \draw[0cell] 
      (0,0) node (y) {Y}
      (120:2) node (a) {F\ol{X}}
      (120:1) node (x) {X}
      (60:2) node (b) {F\ol{Y}}
      ;
      \draw[1cell] 
      (a) edge[swap] node {f_{\ol{X}}} (x)
      (x) edge[swap] node {u} (y)
      (b) edge node {f_{\ol{Y}}} (y)
      (a) edge node (T) {F\ol{u}} (b)
      ;
      \draw[2cell] 
      node[between=y and T at .65, rotate=30, font=\Large] (A) {\Rightarrow}
      (A) node[below right] {\theta_{\ol{u}}}
      ;
    \end{tikzpicture}
    \qquad \mathrm{ and } \qquad
    \begin{tikzpicture}[x=15mm,y=13mm,baseline={(x.base)}]
      \draw[0cell] 
      (0,0) node (y) {Z}
      (120:2) node (a) {F\ol{Y}}
      (120:1) node (x) {Y}
      (60:2) node (b) {F\ol{Z}}
      ;
      \draw[1cell] 
      (a) edge[swap] node {f_{\ol{Y}}} (x)
      (x) edge[swap] node {v} (y)
      (b) edge node {f_{\ol{Z}}} (y)
      (a) edge node (T) {F\ol{v}} (b)
      ;
      \draw[2cell] 
      node[between=y and T at .65, rotate=30, font=\Large] (A) {\Rightarrow}
      (A) node[below right] {\theta_{\ol{v}}}
      ;
    \end{tikzpicture}
    \end{equation}
    Pasting these together and composing with $F^2_{\ol{v},\ol{u}}$, we
    obtain a 1-cell in $F \dar Z$
    \[
    (\ol{v} \circ \ol{u}, \theta')\cn (\ol{X}, v (u f_{\ol{X}})) \to (\ol{Z}, f_{\ol{Z}}),
    \]
    where $\theta'$ is given by the following pasting diagram.
    \begin{equation}\label{Gdef-10}
      \begin{tikzpicture}[x=13mm,y=11mm,baseline={(x.base)}]
        \draw[0cell] 
        (-60:.25) node (y) {Y}
        (120:2.5) node (a) {F\ol{X}}
        (120:1.25) node (x) {X}
        (60:2) node (b) {F\ol{Y}}
        (-60:1.75) node (z) {Z}
        (z) ++(60:4.25) node (c) {F\ol{Z}}
        ;
        \draw[1cell] 
        (a) edge[swap] node {f_{\ol{X}}} (x)
        (x) edge[swap] node {u} (y)
        (b) edge node {f_{\ol{Y}}} (y)
        (a) edge[bend right] node (T) {F\ol{u}} (b)
        (y) edge[swap] node {v} (z)
        (b) edge[bend right] node {F\ol{v}} (c)
        (c) edge node {f_{\ol{Z}}} (z)
        (a) edge[bend left] node (TT) {F(\ol{v} \circ \ol{u})} (c)
        ;
        \draw[2cell] 
        node[between=y and T at .65, rotate=30, font=\Large] (A) {\Rightarrow}
        (A) node[below right] {\theta_{\ol{u}}}
        node[between=y and c at .4, shift={(0,-.1)}, rotate=30, font=\Large] (B) {\Rightarrow}
        (B) node[below right] {\theta_{\ol{v}}}
        node[between=b and TT at .5, shift={(-.05,0)}, rotate=90, font=\Large] (C) {\Rightarrow}
        (C) node[right] {F^2_{\ol{v},\ol{u}}}
        ;
      \end{tikzpicture}
    \end{equation}
    Now by definition, $(\ol{u},\theta_{\ol{u}}) = k^Y_{(\ol{X},uf_{\ol{X}})}$.    
    Therefore by hypothesis \eqref{QA-hypothesis-2} the composite
    $(\ol{v} \circ \ol{u}, \theta')$ is an initial 1-cell
    $(\ol{X},v(uf_{\ol{X}})) \to (\ol{Z},f_{\ol{Z}})$.  We also have
    the component of $k^Z$ at $(\ol{X},(vu)f_{\ol{X}})$.  This is an
    initial 1-cell
    \[
    (\ol{vu}, \theta_{\ol{vu}})\cn (\ol{X},(vu) f_{\ol{X}}) \to
    (\ol{Z}, f_{\ol{Z}})
    \]
    where $\theta_{\ol{vu}}$ is displayed below.
    \begin{equation}\label{Gdef-11}
      \begin{tikzpicture}[x=15mm,y=13mm,baseline={(x.base)}]
        \draw[0cell] 
        (0,0) node (y) {Z}
        (120:2) node (a) {F\ol{X}}
        (120:1) node (x) {X}
        (60:2) node (b) {F\ol{Z}}
        ;
        \draw[1cell] 
        (a) edge[swap] node {f_{\ol{X}}} (x)
        (x) edge[swap] node {vu} (y)
        (b) edge node {f_{\ol{Z}}} (y)
        (a) edge node (T) {F\ol{vu}} (b)
        ;
        \draw[2cell] 
        node[between=y and T at .65, rotate=30, font=\Large] (A) {\Rightarrow}
        (A) node[below right] {\theta_{\ol{vu}}}
        ;
      \end{tikzpicture}
    \end{equation}
    
    Composing $\theta_{\ol{vu}}$ with the associator
    \[
    a_\C^\inv \cn v(uf_{\ol{X}}) \to (uv)f_{\ol{X}}
    \]
    yields another 1-cell
    \[
    (\ol{vu}, a_\C^\inv \theta_{\ol{vu}})\cn (\ol{X}, v(uf_{\ol{X}})) \to (\ol{Z},f_{\ol{Z}}),
    \]
    and therefore there is a unique 2-cell in $\B$
    \[
    (Gv) \circ (Gu) = \ol{v} \circ \ol{u} \to \ol{vu} = G(vu)
    \]
    whose image under $F$ satisfies the ice cream cone condition for
    the triangles \eqref{Gdef-10} and \eqref{Gdef-11}.  We define
    $G^2_{v,u}$ to be this 2-cell.
  \end{itemize}
\end{definition}

\begin{proposition}\label{proposition:G-lax}
  Under the hypotheses of \cref{theorem:Quillen-A-bicat}, the
  assignment on cells defined above specifies a lax functor $G\cn \B
  \to \C$.
\end{proposition}
\begin{proof}

\Gstep{2a} To verify that $G$ defines a functor $\C(X,Y) \to \B(GX,GY)$ for
each $X$ and $Y$, first note that when $\ga = 1_{u}$, then
$1_{\ol{u}}$ satisfies the ice cream cone condition above, and hence by
uniqueness of 2-cells out of an initial 1-cell, we have
\[
G1_u = \ol{(1_u)} = 1_{(\ol{u})} = 1_{Gu}.
\]
Now we turn to functoriality with respect to vertical composition of
2-cells.  Consider a pair of composable 2-cells
\[
u_0 \fto{\ga} u_1 \fto{\de} u_2
\]
between 1-cells $u_0, u_1, u_2\in \C(X,Y)$.  We will show that the
chosen lift $G(\de\ga) = \ol{\de\ga}$ is equal to the composite
\[
(G\de) \circ (G\ga) = \ol{\de} \circ \ol{\ga}.
\]
To do this, we note that $(\ol{u_0},\theta_0)$ is an initial 1-cell
and therefore we simply need to observe that $\ol{\de} \circ \ol{\ga}$
satisfies the ice cream cone condition for $\ol{\de \ga}$.  Then the
uniqueness of 2-cells from $(\ol{u_0},\theta_0)$ to
$(\ol{Y},f_{\ol{Y}})$ will imply the result.  This is done by the four
pasting diagrams below.  The first equality follows by functoriality of $F$:
we have $(F\ol{\de}) (F\ol{\ga}) = F(\ol{\de} \circ \ol{\ga})$.  The next
two equalities follow by the conditions for $\ol{\ga}$ and $\ol{\de}$
individually.
\begin{equation}\label{Gdef-5}
  \begin{tikzpicture}[x=18mm,y=18mm,baseline={(a.base)},scale=.9]
    \def\w{2.3} 
    \def\h{-2.25} 
    \def\m{.5} 
    \draw[font=\Large] (\w/2+\m,0) node (eq) {=}; 
    \newcommand{\boundary}{
      \draw[0cell] 
      (0,-.25) node (y) {Y}
      (120:2) node (a) {F\ol{X}}
      (150:1) node (x) {X}
      (60:2) node (b) {F\ol{Y}}
      ;
      \draw[1cell] 
      (a) edge[swap] node {f_{\ol{X}}} (x)
      (b) edge[bend left] node (fy) {f_{\ol{Y}}} (y) 
      (x) edge[swap, bend right=70] node {u_0} (y)
      (a) edge[bend left=70] node (T) {F\ol{u_2}} (b) 
      (a) edge[swap, bend right=70] node (B) {F\ol{u_0}} (b) 
      ;
    }
    \begin{scope}[shift={(0,\h/2)}]
      \boundary
      \draw[1cell] 
      ;
      \draw[2cell]
      node[between=y and B at .6, rotate=50, font=\Large] (TH0) {\Rightarrow}
      (TH0) node[left, shift={(-.05,.05)}] {\theta_{0}}
      node[between=B and T at .5, rotate=90, font=\Large] (Fga) {\Rightarrow}
      (Fga) node[right] {F(\ol{\de} \circ \ol{\ga})}
      ;
    \end{scope}      
    \begin{scope}[shift={(\w+\m+\m,\h/2)}]
      \boundary
      \draw[1cell] 
      (a) edge[bend right=10] node[pos=.2] (M) {F\ol{u_1}} (b) 
      ;
      \draw[2cell]
      node[between=y and B at .6, rotate=50, font=\Large] (TH0) {\Rightarrow}
      (TH0) node[left, shift={(-.05,.05)}] {\theta_{0}}
      node[between=B and T at .25, rotate=90, font=\Large] (Fga) {\Rightarrow}
      (Fga) node[right] {F\ol{\ga}}
      node[between=y and T at .85, rotate=90, font=\Large] (Fde) {\Rightarrow}
      (Fde) node[right] {F\ol{\de}}
      ;
    \end{scope}      
  \end{tikzpicture}
\end{equation}
\begin{equation}\label{Gdef-6}
  \begin{tikzpicture}[x=18mm,y=18mm,baseline={(a.base)},scale=.9]
    \def\w{2.3} 
    \def\h{-2.25} 
    \def\m{.5} 
    \draw[font=\Large] (\w/2+\m,0) node (eq) {=}; 
    \draw[font=\Large] (-\w/2-\m,0) node (eq) {=}; 
    \newcommand{\boundary}{
      \draw[0cell] 
      (0,-.25) node (y) {Y}
      (120:2) node (a) {F\ol{X}}
      (150:1) node (x) {X}
      (60:2) node (b) {F\ol{Y}}
      ;
      \draw[1cell] 
      (a) edge[swap] node {f_{\ol{X}}} (x)
      (b) edge[bend left] node (fy) {f_{\ol{Y}}} (y) 
      (x) edge[swap, bend right=70] node (u0) {u_0} (y)
      (x) edge[bend right=10] node[pos=.3] (u1) {u_1} (y) 
      (a) edge[bend left=70] node (T) {F\ol{u_2}} (b) 
      ;
      \draw[2cell] 
      node[between=u0 and u1 at .55, rotate=50, font=\Large] (ga) {\Rightarrow}
      (ga) node[above left] {\ga}
      ;
    }
    \begin{scope}[shift={(0,\h/2)}]
      \boundary
      \draw[1cell] 
      (a) edge[bend right=10] node[pos=.2] (M) {F\ol{u_1}} (b) 
      ;
      \draw[2cell] 
      node[between=y and T at .85, rotate=90, font=\Large] (Fde) {\Rightarrow}
      (Fde) node[right] {F\ol{\de}}
      node[between=y and T at .5, shift={(.05,-.1)}, rotate=50, font=\Large] (TH1) {\Rightarrow}
      (TH1) node[below, shift={(.05,-.05)}] {\theta_{1}}
      ;
    \end{scope}      
    \begin{scope}[shift={(\w+\m+\m,\h/2)}]
      \boundary
      \draw[1cell] 
      (x) edge[bend left=70] node[pos=.7] (u2) {u_2} (y) 
      ;
      \draw[2cell] 
      node[between=y and T at .5, shift={(.05,.2)}, rotate=50, font=\Large] (TH2) {\Rightarrow}
      (TH2) node[below, shift={(.05,-.05)}] {\theta_{2}}
      node[between=u0 and u1 at .55, rotate=50, font=\Large] (ga) {\Rightarrow}
      (ga) node[above left] {\ga}
      node[between=u0 and u2 at .70, rotate=50, font=\Large] (de) {\Rightarrow}
      (de) node[below right] {\de}
      ;
    \end{scope}
  \end{tikzpicture}
\end{equation}
Since $\ol{\de \ga}$ is the unique 2-cell satisfying
this condition, we must have $\ol{\de \ga} = \ol{\de} \circ \ol{\ga}$.
Therefore the definition of $G$ is functorial with respect to vertical
composition of 2-cells.

\Gstep{2b}  Naturality of $G^0$ is vacuous.  Naturality of $G^2$
follows because $(\ol{v} \circ \ol{u}, \theta')$ shown in \eqref{Gdef-10} is
initial.  Therefore given $\ga\cn u_0 \to u_1$ and $\de\cn v_0 \to v_1$, the two
composites
\[
(\ol{v_0} \circ \ol{u_0},\theta'_0) \to (\ol{v_1u_1},\theta_{\ol{v_1u_1}})
\]
are equal.

\Gstep{2c} Now we need to verify the lax associativity axiom
\eqref{f2-bicat} and two lax unity axioms \eqref{f0-bicat} for $G$.  We
show that each of the 2-cells involved is the projection to $\B$ of a
2-cell in a lax slice category, and that each composite in the
diagrams is a 2-cell whose source is initial.  Thus we conclude in
each diagram that the two relevant composites are equal.

First, let us consider the lax associativity hexagon \eqref{f2-bicat}
for $G^2$ and the associators.
Given a composable triple
\[
  W \fto{s} X \fto{u} Y \fto{v} Z
\]
we need to show that the following diagram commutes
\begin{equation}\label{Gdef-12}
  \begin{tikzpicture}[x=20mm,y=20mm,baseline=(X.base)]
    \draw[0cell] 
    (0,0) node (3) {((Gv)(Gu))\, (Gs)}
    (3) ++(60:1) node (2) {(Gv)\, ((Gu)(Gs))}
    (3) ++(-60:1) node (4) {(G(vu))\, (Gs)}
    (3) ++(3,0) node (0) {G(v\,(us))}
    (0) ++(120:1) node (1) {(Gv)\, (G(us))}
    (0) ++(-120:1) node (5) {G((vu)\, s)}
    ;
    \draw[1cell] 
    (3) edge node (X) {a_\B} (2)
    (2) edge node {1 * G^2} (1)
    (1) edge node {G^2} (0)
    (3) edge[swap] node {G^2 * 1} (4) 
    (4) edge[swap] node {G^2} (5) 
    (5) edge[swap] node {Ga_\C} (0) 
    ;
  \end{tikzpicture}
\end{equation}
where $a_\B$ and $a_\C$ denote the associators in $\B$ and $\C$
respectively. To do this, we observe that this entire diagram is the
projection to $\B$ of the following diagram in $F \dar Z$, where we
use two key details from the description in \cref{defprop:lax-slice}:
\begin{itemize}
\item The horizontal composition of 2-cells in $F\dar Z$
  (namely, the whiskering of 2-cells by 1-cells) is given by
  horizontal composition in $\B$.
\item The associator in $F \dar Z$ is given by $(a_\B)$.
\end{itemize}
\begin{equation}\label{Gdef-13}
  \begin{tikzpicture}[x=25mm,y=20mm,baseline=(X.base)]
    \draw[0cell] 
    (0,0) node (3) {
      \big(
      (\ol{v}, \theta_{\ol{v}}) (\ol{u}, 1_v * \theta_{\ol{u}})
      \big)
      \, (\ol{s}, 1_{vu} * \theta_{\ol{s}})
    }
    (3) ++(70:1) node (2) {
      (\ol{v}, \theta_{\ol{v}}) \, \big(
      (\ol{u}, 1_v * \theta_{\ol{u}}) (\ol{s}, 1_v * (1_u * \theta_{\ol{s}}))
      \big)
    }
    (3) ++(-70:1) node (4) {
      (\ol{vu}, \theta_{\ol{vu}})\, (\ol{s}, 1_{vu} * \theta_{\ol{s}})
    }
    (3) ++(3,0) node (0) {
      (\ol{v\,(us)}, \theta_{\ol{v(us)}})
    }
    (0) ++(180-70:1) node (1) {
      (\ol{v}, \theta_{\ol{v}}) \, (\ol{us}, 1_v * \theta_{\ol{us}})
    }
    (0) ++(180+70:1) node (5) {
      (\ol{(vu)\, s}, \theta_{\ol{(vu)s}})
    }
    ;
    \draw[1cell] 
    (3) edge node[pos=.4] (X) {(a_\B)} (2)
    (2) edge node {(1 * G^2)} (1)
    (1) edge node[pos=.6] {(G^2)} (0)
    (3) edge[swap] node[pos=.3] {(G^2 * 1)} (4) 
    (4) edge[swap] node {(G^2)} (5) 
    (5) edge[swap] node[pos=.7] {(\ol{a_\C})} (0) 
    ;
  \end{tikzpicture}
\end{equation}
Now the 1-cells $(\ol{v},\theta_{\ol{v}})$,
$(\ol{u},\theta_{\ol{u}})$, and $(\ol{s},\theta_{\ol{s}})$ are defined
to be components of $k^X$, $k^Y$, and $k^Z$, respectively.  We have
$(F \dar u)(\ol{s}, \theta_{\ol{s}}) = (\ol{s}, 1_u *
\theta_{\ol{s}})$ and therefore
\[
(\ol{u},\theta_{\ol{u}}) (\ol{s},1_u * \theta_{\ol{s}})
\]
is initial by hypothesis \eqref{QA-hypothesis-2} and
\cref{lemma:preserves-initial-1-cells}.  The strict functor $F \dar v$
sends this composite to
\[
(\ol{u},1_v * \theta_{\ol{u}}) (\ol{s}, 1_v * (1_u * \theta_{\ol{s}})),
\]
so the upper-left corner of the hexagon is initial by
hypothesis \eqref{QA-hypothesis-2} and \cref{lemma:preserves-initial-1-cells} again.
Since $a_\B$ is an isomorphism, this implies that
\[
\big(
(\ol{v}, \theta_{\ol{v}}) (\ol{u}, 1_v * \theta_{\ol{u}})
\big)
\, (\ol{s}, 1_{vu} * \theta_{\ol{s}})
\]
is also an initial 1-cell.  Therefore the two composites around the diagram are equal
and consequently their projections to $\B$ are equal.

\Gstep{2d} Next we consider the lax unity axioms \eqref{f0-bicat} for a 1-cell $u\cn X \to
Y$.  We use subscripts $\B$ or $\C$ to denote the respective unitors.
As with the lax associativity axiom, the necessary diagrams
are projections to $\B$ of diagrams in $F \dar Y$, each of whose
source 1-cell is initial.  Therefore the diagrams in $F \dar Y$
commute and hence their projections to $\B$ commute.
\begin{equation}\label{Gdef-14}
  \begin{tikzpicture}[x=23mm,y=18mm,baseline={(0,1)}]
    \draw[0cell] 
    (0,0) node (a) {(Gu) (1_{GX})}
    (.25,1) node (b) {(Gu) (G1_X)}
    (1.25,1) node (c) {G(u 1_X)}
    (1.5,0) node(d) {Gu}
    ;
    \draw[1cell] 
    (a) edge node[pos=.4] {1*G^0} (b)
    (b) edge node {G^2} (c)
    (c) edge node[pos=.6] {Gr_\C} (d)
    (a) edge[swap] node {r_\B} (d)
    ;
  \end{tikzpicture}
  \qquad 
  \begin{tikzpicture}[x=23mm,y=18mm, baseline={(0,1)}]
    \draw[0cell] 
    (0,0) node (a) {(1_{GY}) (Gu)}
    (.25,1) node (b) { (G1_Y) (Gu)}
    (1.25,1) node (c) {G(1_Y u)}
    (1.5,0) node(d) {Gu}
    ;
    \draw[1cell] 
    (a) edge node[pos=.4] {G^0*1} (b)
    (b) edge node {G^2} (c)
    (c) edge node[pos=.6] {G(\ell_\C)} (d)
    (a) edge[swap] node {\ell_\B} (d)
    ;
  \end{tikzpicture}
\end{equation}
This completes the proof that $G$ is a lax functor $\C \to \B$.
\end{proof}

\begin{proof}[Proof of \cref{theorem:Quillen-A-bicat}]
Now we turn to the transformations
\[
  \eta\cn \Id_\B \to GF \quad \mathrm{ and } \quad \epz\cn FG \to \Id_\C.
\]

\Gstep{3a} The components of $\epz$ are already defined in the
construction of $G$: given an object $X$, we define $\epz_X =
f_{\ol{X}}$, the 1-cell part of the inc-lax terminal object $(\ol{X},
f_{\ol{X}})$.  For a 1-cell $u$, we define $\epz_u = \theta_{\ol{u}}$,
the 2-cell part of the initial 1-cell
\[
(\ol{u},\theta_{\ol{u}}) \cn (\ol{X},u f_{\ol{x}}) \to (\ol{Y},f_{\ol{Y}}).
\]

To define the components of $\eta$, suppose $A$ and $B$ are objects of $\B$
and suppose $p\cn A \to B$ is a 1-cell between them.  Then $(A,1_A)$ defines
an object of $F \dar FA$.  Therefore there is an initial 1-cell
\[
([A],\theta_{[A]})\cn(A,1_A) \to (\ol{FA}, f_{\ol{FA}})
\]
to the inc-lax terminal object in $F \dar FA$.  We define
\[
\eta_A = [A] \cn A \to \ol{FA} = G(FA).
\]
Given a 1-cell $p\cn A \to B$ in $\B$ we have two different 1-cells
in $F \dar FB$
\[
(A,1_{FA} (Fp)) \to (\ol{FB},f_{\ol{FB}}).
\]
One of these is the composite
\begin{equation}\label{comp-1}
(\ol{Fp}, \theta_{\ol{Fp}}) \circ (F \dar Fp)(\eta_A,\theta_{\eta_A}),
\end{equation}
and note that this is initial by hypothesis \eqref{QA-hypothesis-2}
and \cref{lemma:preserves-initial-1-cells}.  The other 1-cell is the composite
\begin{equation}\label{comp-2}
(\eta_B,\theta_{\eta_B}) \circ (p,\upsilon),
\end{equation}
where $\upsilon$ denotes a composite of unitors.  The 2-cell
components of the composites \eqref{comp-1} and \eqref{comp-2} are
given, respectively, by the two pasting diagrams below.
\begin{equation}\label{Gdef-15}
  \begin{tikzpicture}[x=11.1mm,y=15mm,baseline={(F2.base)}]
    \draw[0cell] 
    (0,0) node (z) {FB}
    (90:2) node (xt) {F(\ol{FA})}
    (45:2.828) node (zt) {F(\ol{FB})}
    (135:2.828) node (xt0) {FA}
    (135:1.414) node (y) {FA}
    ;
    \draw[1cell] 
    (xt0) edge node (F1xt) {F\eta_A} (xt)
    (xt0) edge[swap] node {1_{FA}} (y)
    (y) edge[swap] node {Fp} (z)
    (xt0) edge[bend left=60, looseness=1.1] node (U) {F(\ol{Fp}\, \eta_A)} (zt)
    
    (xt) edge node (fFa) {f_{\ol{FA}}} (y)
    (zt) edge node[pos=.45] (fFb) {f_{\ol{FB}}} (z)
    (xt) edge node (T) {F(\ol{Fp})} (zt)
    ;
    \draw[2cell] 
    node[between=fFb and xt at .6, rotate=30, font=\Large] (A) {\Rightarrow}
    (A) node[below right] {\theta_{\ol{Fp}}}
    node[between=y and F1xt at .4, shift={(.108,0.08)}, rotate=30, font=\Large] (B) {\Rightarrow}
    (B) node[above left] {\theta_{\eta_A}}
    node[between=xt and U at .5, shift={(.1,0)}, rotate=90, font=\Large] (D) {\Rightarrow}
    (D) node[right] (F2) {F^2}
    ;
  \end{tikzpicture}
  \qquad
  \begin{tikzpicture}[x=11mm,y=15mm,baseline={(F2.base)}]
    \draw[0cell] 
    (0,0) node (z) {FB}
    (90:2) node (xt) {FB}
    (45:2.828) node (zt) {F(\ol{FB})}
    (135:2.828) node (xt0) {FA}
    (135:1.414) node (y) {FA}
    ;
    \draw[1cell] 
    (xt0) edge node (F1xt) {Fp} (xt)
    (xt0) edge[swap] node {1_{FA}} (y)
    (y) edge[swap] node {Fp} (z)
    (xt0) edge[bend left=60, looseness=1.1] node (U) {F(\eta_B\, p)} (zt)
    
    (xt) edge node {1_{FB}} (z)
    (zt) edge node[pos=.45] (fFb) {f_{\ol{FB}}} (z)
    (xt) edge node (T) {F(\eta_B)} (zt)
    ;
    \draw[2cell] 
    node[between=fFb and xt at .6, rotate=30, font=\Large] (A) {\Rightarrow}
    (A) node[below right] {\theta_{\eta_B}}
    node[between=y and xt at .45, rotate=30, font=\Large] (B) {\Rightarrow}
    (B) node[above left] {\upsilon}
    node[between=xt and U at .5, shift={(.1,0)}, rotate=90, font=\Large] (D) {\Rightarrow}
    (D) node[right] (F2) {F^2}
    ;
  \end{tikzpicture}
\end{equation}
Since the diagram at left in
\eqref{Gdef-15} corresponds to an initial 1-cell, we therefore have a
unique 2-cell $(G(Fp)) \eta_A \to \eta_B p$ in $\B$ whose image under
$F$ satisfies the ice cream cone condition with respect to the two
outermost triangles in \eqref{Gdef-15}.  We take $\eta_p$ to be this 2-cell.

\Gstep{3b} Naturality of the components $\epz_u$ with respect to
2-cells $\ga\cn u_0 \to u_1$ is precisely the condition in
\eqref{Gdef-4} defining $G\ga = \ol{\ga}$.  Naturality of the
components $\eta_p$ with respect to 2-cells $\om\cn p_0 \to p_1$
follows because the source 1-cell shown at left in \eqref{Gdef-15} is initial.

\Gsteps{3c}{3d} The lax transformation axioms for $\epz$ and $\eta$ follow immediately
from the inc-lax terminal conditions for $k^X$; the unit axiom follows from the unit
condition for $k^X$, and the 2-cell axiom follows from uniqueness of
2-cells out of an initial 1-cell.
\end{proof}

\section{The Whitehead Theorem for Bicategories}\label{sec:Whitehead-bicat}

In this section we apply the bicategorical Quillen Theorem A
(\ref{theorem:Quillen-A-bicat}) to prove the Bicategorical Whitehead
Theorem.

\begin{theorem}[Whitehead Theorem for Bicategories]\label{theorem:whitehead-bicat}
  A pseudofunctor of bicategories $F\cn \B \to \C$ is a biequivalence
  if and only if $F$ is
  \begin{enumerate}
  \item essentially surjective on objects;
  \item essentially full on 1-cells; and
  \item fully faithful on 2-cells.
  \end{enumerate}
\end{theorem}
\begin{proof}
  One implication is immediate: if $F$ is a biequivalence with inverse
  $G$, then the internal equivalence $FG \hty \Id_\C$ implies that $F$
  is essentially surjective on objects.  
  \cref{lemma:biequiv-implies-local-equiv} proves that $F$ is
  essentially full on 1-cells and fully faithful on 2-cells.
  
  If $F$ is essentially surjective, essentially full, and fully faithful, then
  \cref{proposition:lax-slice-lax-terminal,lemma:lax-slice-change-fiber}
  show that the lax slices have inc-lax terminal objects and that the
  strict functors $F \dar u$ preserve initial components.
  Therefore we apply \cref{theorem:Quillen-A-bicat} to obtain $G\cn\C
  \to \B$ together with $\epz$ and $\eta$.

  Moreover, the proof of \cref{proposition:lax-slice-lax-terminal} shows that
  the components $\epz_X = f_{\ol{X}}$ and $\epz_u = \theta_{\ol{u}}$
  are invertible.  Likewise, if the constraints $F^0$ and $F^2$ are
  invertible then the ice cream cone conditions for $F(G^0)$ and
  $F(G^2)$, together with invertibility of the $\theta_{\ol{u}}$,
  imply that $F(G^0)$ and $F(G^2)$ are invertible.  Thus $G^0$ and
  $G^2$ are invertible because $F$ is fully faithful on 2-cells and
  therefore reflects isomorphisms.  Therefore $G$ is a pseudofunctor.

  Likewise in the construction of $\eta_A$ via
  \cref{proposition:lax-slice-lax-terminal}, we note that
  $\theta_{\eta_A}$ and $f_{\ol{FA}}$ are both invertible, so
  $F\eta_A$ is invertible.  The assumption that $F$ is essentially
  surjective on 1-cells and fully faithful on 2-cells implies that $F$
  reflects invertibility of 1-cells, and therefore $\eta_A$ is
  invertible.  Similarly, the construction of $\eta_p$ under these
  hypotheses implies that $F(\eta_p)$ is invertible and hence $\eta_p$
  is invertible.

  Now $\eta$ and $\epz$ are strong transformations with invertible
  components.  Therefore by
  \cref{proposition:adjoint-equivalence-componentwise} we conclude
  that $\eta$ and $\epz$ are invertible strong transformations.  Thus
  $F$ and $G$ are inverse biequivalences.
\end{proof}
